\documentclass[11pt]{article}
	\usepackage{amsfonts,amssymb,amsthm,amsmath,cite,bbm}
	\usepackage{mathrsfs}

	\newcommand{\R}{\mathbb R} \newcommand{\N}{\mathbb N}

	\newcommand{\os}[1]{\mathop{\langle {#1} \rangle}} 
	\newcommand{\oid}[1]{\mathop{[ {#1} ]}}

\newcommand{\vm}{\operatorname{m}} 

\newcommand{\cK}{{\mathcal K}} \newcommand{\cP}{{\mathcal P}}

\newcommand{\cL}{{\mathcal L}} \newcommand{\cS}{{\mathcal S}}
\newcommand{\Kn}{\cK^n}
\newcommand{\Ko}{\Kn_{0}}
\newcommand{\Pn}{\cP^n}
\newcommand{\Po}{\cP^n_{0}}

\newcommand{\cH}{{\mathcal H}}

\newcommand{\sn}{{\mathbb S}^{n-1}}
\newcommand{\Rn}{\R^n}

\renewcommand{\d}{\,\mathrm{d}}
\newcommand{\e}{\varepsilon}
\renewcommand{\H}{\mathcal{H}}

\newcommand{\mx}{\mathbin{\vee}} 
\newcommand{\mn}{\mathbin{\wedge}} 

\newcommand{\conv}{\operatorname{conv}}
\newcommand{\elim}{\operatorname{epi-lim}}

\newcommand{\dom}{\operatorname{dom}}
\newcommand{\epi}{\operatorname{epi}}
\newcommand{\pos}{\operatorname{pos}}
\newcommand{\supp}{\operatorname{supp}}

\newcommand{\Ind}{\mathrm{I}}

\newcommand{\Cos}[1]{\operatorname{\mathscr C}\!#1}

\newcommand{\oY}{\operatorname{Y}} 
\newcommand{\oZ}{\operatorname{Z}}
 	
\newcommand{\oD}{\operatorname{D}} 
\newcommand{\oM}{\operatorname{\rm M}} 
\newcommand{\oP}{\operatorname{\Pi}} 

\newcommand{\meas}{{\mathcal M}(\sn)} 
\newcommand{\mease}{{\mathcal M}_e(\sn)} 
\newcommand{\sW}{{W^{1,1}(\Rn)}} 

\newcommand\sln{\operatorname{SL}(n)}
\newcommand\slnbf{{\bf\operatorname{\bf SL}(n)}}
\newcommand\gln{\operatorname{GL}(n)}

\newcommand{\eqnref}[1]{(\ref{#1})}

\newtheorem{Theorem}{Theorem}
\newtheorem{lemma}{Lemma}[section]
\newtheorem{theorem}[lemma]{Theorem}

\newcommand{\C}[1]{\operatorname{Conv}(\R^{#1})}
\newcommand{\CH}{\operatorname{Conv}(E)}
\newcommand{\CV}{\operatorname{Conv}(\Rn)}
\newcommand{\CVb}{\operatorname{\bf Conv}(\Rn)}

\newcommand{\QC}{\operatorname{QC}(\Rn)}
\newcommand{\BV}{\operatorname{BV}(\Rn)} 

\newcommand{\eto}{\stackrel{epi}{\longrightarrow}}

\newcommand{\D}[1]{{D}^{#1}(\R)}
\newcommand{\Dn}{\D{n-2}} %
\newcommand{\Dz}{\D{0}} %

\newcommand{\proj}{\operatorname{proj}}
\newcommand{\um}{u^{\scriptscriptstyle -}}

\title{Minkowski Valuations on Convex Functions}
\author{Andrea Colesanti,  Monika Ludwig and Fabian Mussnig}
\date{}

\begin{document}

\maketitle

\begin{abstract}
A classification of $\sln$ contravariant  Minkowski valuations on convex functions and a characterization of the projection body operator are established. The associated LYZ measure is characterized. In addition, a new $\sln$ covariant Minkowski valuation on convex functions is defined and characterized.

\bigskip


\end{abstract}

Several important norms on $\Rn$ or convex bodies  (that is, convex compact sets) in $\Rn$  have been associated to functions $f:\R^n\to \R$. On  the Sobolev space $\sW$  (that is,  the space of functions $f\in L^1(\R^n)$ with weak gradient $\nabla f \in L^1(\R^n)$), Gaoyong Zhang \cite{Zhang99} defined the \emph{projection body} $\oP\os{f}$.  Using the support function of a convex body $K$ (where $h(K,y)=\max\{ y\cdot x: x\in K\}$ with $y\cdot x$ the standard inner product of $x,y\in\R^n$) to describe $K$, this convex body is given by
$$h(\oP\os{f}, y)= \int_{\Rn} \lvert y\cdot \nabla f(x)\rvert\d x$$
for $y\in\R^n$. The operator  that associates to $f$ the convex body $\oP\os{f}$ is easily seen to be $\sln$ contravariant,  where, in general,   an operator $\oZ$ defined on some space of functions $f:\Rn\to\R$ and with values in the space of convex bodies, $\Kn$, in $\Rn$  is \emph{$\sln$ contravariant} if
$\oZ(f\circ\phi^{-1})=\phi^{-t}\oZ(f)$ 
for every function $f$ and $\phi\in\sln$.  Here $\phi^{-t}$ is the inverse of the transpose of $\phi$. The projection body of $f$ turned out to be critical in Zhang's affine Sobolev inequality \cite{Zhang99}, which is  a sharp affine isoperimetric inequality essentially stronger than the $L^1$~Sobolev inequality. The convex body $\oP\os{f}$ is the classical  projection body (see Section~\ref{tools} for the definition) of another convex body $\os{f}$, which is the unit ball of the so-called optimal Sobolev norm of $f$ and was introduced by  Lutwak, Yang and Zhang \cite{LYZ2006}. The operator $f\mapsto \os{f}$ is called the \emph{LYZ operator}. It is $\sln$ covariant, where, in general,  an operator $\oZ$ defined on some space of functions $f:\Rn\to\R$ and with values in $\Kn$ is \emph{$\sln$\! covariant} if  $\oZ(f\circ\phi^{-1})=\phi\oZ(f)$ 
for every function $f$ and $\phi\in\sln$.
See also \cite{CLYZ, HaberlSchuster09, Haberl:Schuster:Xiao, LYZ2002b, LXZ, AMJV, Tuo_Wang:PSP}.

In \cite{Ludwig:sobval}, 
a characterization of the operators
$f\mapsto \oP\os{f}$ and $f \mapsto \os{f}$ as $\sln$ contravariant and $\sln$ covariant valuations on $\sW$ was established.
Here, a function $\oZ$ defined on a lattice $(\cL,\mx, \mn)$ and taking values in an abelian semigroup is called a {\em valuation} if
\begin{equation}\label{valuation}
\oZ(f\mx g)+\oZ(f\mn g)=\oZ(f) +\oZ (g)
\end{equation}
for all $f,g\in \cL$. 
A function $\oZ$ defined on some subset $\cS$ of $\cL$ is called a valuation on $\cS$ if 
\eqnref{valuation} holds whenever $f,g, f\mx g, f\mn g\in \cS$.
For $\cS$ the space of convex bodies, $\Kn$, in $\R^n$ with $\mx$ denoting union and $\mn$ intersection, the notion of  valuation is classical and it was the key ingredient in  Dehn's solution of Hilbert's Third Problem in 1901 (see \cite{Hadwiger:V,Klain:Rota}). Interesting new valuations keep arising (see, for example, \cite{HLYZ_acta} and see  \cite{Alesker99,Alesker01,Bernig:Fu, Haberl_sln,  Ludwig:Reitzner2,  HaberlParapatits_moments, Haberl:Parapatits_centro, LiMa, AbardiaWannerer} for some recent results on valuations on convex bodies). More recently, valuations started to be studied on function spaces. When $\cS$ is a space of real valued functions, then we take  $u\mx v$ to be the pointwise maximum of $u$ and $v$ while $u\mn v$ is the pointwise minimum. For Sobolev spaces \cite{Ludwig:sobval, Ludwig:Fisher, Ma2016} and $L^p$ spaces \cite{Tsang:Minkowski, Tsang:Lp, Ludwig:MM} complete classifications for valuations intertwining the $\sln$ were established. See also \cite{Ludwig:survey, Kone, Ober2014, Tuo_Wang_semi,  BaryshnikovGhristWright, ColesantiLombardi, ColesantiLombardiParapatits, CavallinaColesanti,Alesker_convex}.

The aim of this paper is to establish a classification of $\sln$~covariant and of $\,\sln$~contra\-variant Minkowski valuations on convex functions. Let $\CV$ denote the space of convex functions $u: \R^n \to (-\infty, +\infty]$ which are {proper}, lower semicontinuous and coercive. Here a function is \emph{proper} if it is not identically $+\infty$ and it is \emph{coercive} if 
\begin{equation}\label{eq:cn_coercive}
\lim_{\lvert x\rvert \to+ \infty} u(x)=+\infty,
\end{equation}
where $\lvert x\rvert$ is the Euclidean norm of $x$.
The space $\CV$ is one of the standard spaces in convex analysis and here it is equipped with the topology associated to epi-convergence (see Section~\ref{tools}). An operator $\oZ: \cS \to \Kn$ is a \emph{Minkowski valuation} if (\ref{valuation}) holds with the addition  on $\Kn$ being Minkowski addition (that is, $K+L=\{x+y: x\in K, y\in L\}$ for $K,L\in \Kn$). The projection body operator is an $\sln$ contravariant Minkowski valuation on $\sW$ while the LYZ operator itself is not a Minkowski valuation (for $n\ge 3$) but a Blaschke valuation (see Section \ref{tools} for the definition). 

In our first result,  we establish a classification of $\sln$ contravariant Minkowski valuations on $\CV$. To this end, we extend the definition of projection bodies  to functions $\zeta\circ u$ with $u\in\CV$ and $\zeta\in \Dn$, where, for $k\ge0$, 
\begin{equation*}
\D{k}=\big\{\zeta\in C(\R)\,: \, \zeta\geq 0, \, \zeta \text{ is decreasing and  } \int_0^\infty t^{k} \zeta(t)\d t <\infty\big\}.
\end{equation*}
We call an operator  $\oZ:\CV\to \Kn$ \emph{translation invariant} if
$\oZ(u\circ\tau^{-1})=\oZ(u)$ 
for every $u\in\CV$ and every translation $\tau:\Rn\to\Rn$. Let $n\ge 3$.

\begin{Theorem}\label{contravariant}
A function $\,\oZ :\CV\to \Kn$ is a  continuous, monotone, $\sln$ contra\-variant and  translation invariant Minkowski valuation if and only if there exists $\zeta\in\Dn$ such that
 $$\oZ (u) = \oP \os{\zeta \circ u}$$
for every $u\in\CV$. 
\end{Theorem}

\noindent
Here  $\oZ:\CV\to \Kn$  is  \emph{decreasing} if $\oZ(u)\subseteq \oZ(v)$ for all $u,v\in\CV$ such that $u\geq v$. It is \emph{increasing} if $\oZ(v)\subseteq \oZ(u)$ for all $u,v\in\CV$ such that $u\geq v$. It is \emph{monotone} if it is decreasing or increasing.

\goodbreak

While on the Sobolev space $\sW$ a classification of $\sln$ contravariant Minkowski valuations was established in \cite{Ludwig:sobval},  no classification of $\sln$ covariant Minkowski valuations was obtained on $\sW$. On $\CV$, we  introduce new $\sln$ covariant Minkowski valuations and establish a classification theorem. For $u\in\CV$ and $\zeta\in\Dz$, define the \emph{level set body} $\oid{\zeta \circ u}$ by
$$
h(\oid{\zeta \circ u}, y)=  \int_0^{+\infty} h(\{\zeta\circ u\geq t\},y) \d t
$$
for $y\in\Rn$. Hence the level set body is a Minkowski average of the level sets. 
Let $n\ge 3$.

\begin{Theorem}\label{covariant}
An operator $\,\oZ :\CV\to \Kn$ is a continuous, monotone, $\sln$ covariant and  translation invariant Minkowski valuation if and only if there exists $\zeta\in\Dz$ such that
 $$\oZ (u) = \oD \oid{\zeta \circ u}$$
for every $u\in\CV$. 
\end{Theorem}

\noindent
Here, the difference body, $\oD K$, of a convex body $K$ is defined as $\oD K =K + (-K)$, where $h(-K, y)= h(K,-y)$ for $y\in\R^n$ is the support function of the central reflection of $K$.

While on $\sW$ a classification of $\sln$ covariant Blaschke valuations was established in \cite{Ludwig:sobval},  on  $\CV$ we obtain a more general classification of $\sln$ contravariant measure-valued valuations. For $K\in\Kn$, let $S(K,\cdot)$ denote its surface area measure (see Section \ref{tools}) and let $\mease$ denote the space of finite even Borel measures on $\sn$. See Section \ref{measdef} for the definition of monotonicity and $\sln$ contravariance of  measures. Let $n\ge 3$.

\begin{Theorem}\label{thm:measure}
An operator $\,\oY:\CV\to \mease$ is a weakly continuous, monotone  valuation that is $\,\sln$~contravariant of degree $1$ and translation invariant if and only if there exists $\zeta\in\D{n-2}$ such that
\begin{equation}
\label{eq:even_sam_val}
\oY(u,\cdot)=S( \os{\zeta\circ u},\cdot)
\end{equation}
for every $u\in\CV$.
\end{Theorem}

\noindent
Here, for $\zeta\in\Dn$ and $u\in\CV$, the measure $S(\os{\zeta\circ u}, \cdot)$ is the LYZ measure  of $\zeta\circ u$ (see Section \ref{measdef} for the definition).
The above theorem extends results by Haberl and Parapatits \cite{Haberl:Parapatits_crelle} from convex bodies to convex functions.

\section{Preliminaries}\label{tools}

We collect some properties of convex bodies and convex functions. Basic references are the books by Schneider \cite{Schneider:CB2} and  Rockafellar \& Wets \cite{RockafellarWets}. In addition, we recall definitions and classification results on Minkowski valuations and measure-valued valuations. 

We work in $\R^n$ and denote the canonical basis vectors by $e_1,\dots, e_n$. For a $k$-dimensional linear subspace $E\subset \R^n$, we write $\proj_E: \R^n\to E$ for the orthogonal projection to $E$ and $V_k$ for the $k$-dimensional volume (or Lebesgue measure)  on $E$. Let $\conv(A)$ be the convex hull of $A\subset \R^n$.

\goodbreak
The space of convex bodies, $\Kn$, is equipped with the \emph{Hausdorff metric}, which is given by
\begin{equation*}
\delta(K,L)=\sup\nolimits_{y\in\sn} |h(K,y)-h(L,y)|
\end{equation*}
for $K,L\in\Kn$, where $h(K,y)=\max\{y\cdot x: x\in K\}$ is the support function of $K$ at $y\in\Rn$. 
The subspace of convex bodies in $\R^n$ containing the origin is denoted by $\Ko$. Let $\Pn$ denote the space of convex polytopes in $\R^n$ and $\Po$ the space of convex polytopes containing the origin. All these spaces are equipped with the topology coming from the Hausdorff metric.

For $p\ge 0$, a function $h:\Rn\to \R$ is \emph{p-homogeneous} if $h(t\,z)= t^p\, h(z)$ for $t\ge 0$ and $z\in\Rn$. It is 
 \emph{sublinear} if it is $1$-homogeneous and  $h(y+z)\le h(y) +h(z)$ for $y,z\in\Rn$. Every sublinear function is the support function of a unique convex body.  Note that for the Minkowski sum of  $K,L\in\Kn$, we have
\begin{equation}
\label{eq:spt_fct_addtv}
h(K+L,y)=h(K,y)+h(L,y)
\end{equation}
for $y\in\Rn$.

A second important way to describe a convex body is through its \emph{surface area measure}.
For a Borel set $\omega\subset \sn$ and $K\in\Kn$, the surface area measure $S(K,\omega)$ is the $(n-1)$-dimensional Hausdorff measure of the set of all boundary points of $K$ at which there exists a unit outer normal vector of $\partial K$ belonging to $\omega$. The solution to the Minkowski problem states that a finite Borel measure $\oY$ on $\sn$ is the surface area measure of an $n$-dimensional convex body $K$ if and only if $\oY$ is not concentrated on a great subsphere and $\int_{\sn} u\,d\oY(u)=0$. If such a measure $\oY$ is given, the convex body $K$ is unique up to translation. 

For $n$-dimensional convex bodies $K$ and $L$ in $\Rn$, the \emph{Blaschke sum} is defined as the convex body with surface area measure $S(K,\cdot)+ S(L,\cdot)$ and with centroid at the origin. We call an operator $\oZ: \cS \to \Kn$  a \emph{Blaschke valuation} if (\ref{valuation}) holds with the addition  on $\Kn$ being Blaschke  addition.

\subsection{Convex and Quasi-concave Functions}

We collect results on convex and quasi-concave functions including some results on valuations on convex functions.
To every convex function $u:\Rn\to (-\infty,+\infty]$, there are assigned several convex sets. 
The domain, $\dom u=\{x\in\R^n: u(x)<+\infty\}$, of $u$ is convex and the \emph{epigraph} of $u$,
\begin{equation*}
\epi u =\{(x,y)\in\Rn\times\R: u(x)\leq y\},
\end{equation*}
is a convex subset of $\Rn\times\R$. 
For  $t\in(-\infty,+\infty]$, the \emph{sublevel set},
\begin{equation*}
\{u\leq t\}=\{x\in\Rn:u(x)\leq t\},
\end{equation*}
is convex. For $u\in\CV$, it is also compact. Note that for $u,v\in\CV$ and $t\in\R$, 
\begin{equation}
\label{eq:lvl_sts_val}
\{u\wedge v \leq t\} = \{u\leq t\} \cup \{v\leq t\}\qquad\text{ and }\qquad \{u\vee v\leq t\}= \{u\leq t\} \cap \{v\leq t\},
\end{equation}
where for $u\wedge v\in\CV$ all occurring sublevel sets are either empty or in $\Kn$.

\goodbreak

We equip $\CV$ with the topology associated to epi-convergence.
Here a sequence $u_k: \Rn\to (-\infty, \infty]$ is \emph{epi-convergent}  to $u:\Rn\to (-\infty, \infty]$ if for all $x\in\Rn$ the following conditions hold:
\begin{itemize}
	\item[(i)] For every sequence $x_k$ that converges to $x$,
			\begin{equation*}\label{eq:gc_inf}
				u(x) \leq \liminf_{k\to \infty} u_k(x_k).
			\end{equation*}
	\item[(ii)] There exists a sequence $x_k$ that converges to $x$ such that
			\begin{equation*}\label{eq:gc_sup}
				u(x) = \lim_{k\to\infty} u_k(x_k).
			\end{equation*}
\end{itemize}
\vskip -8pt
In this case we write $u=\elim_{k\to\infty} u_k$ and $u_k \eto u$. We remark that epi-convergence is also called $\Gamma$-convergence.

We require some results connecting epi-convergence and Hausdorff convergence of sublevel sets.
We say that $\{u_k \leq t\} \to \emptyset$ as $k\to\infty$ if there exists $k_0\in\N$ such that $\{u_k \leq t\} = \emptyset$ for all $k\geq k_0$. Also note that
if $u\in\CV$, then 
$$
\inf\nolimits_{\R^n}u=\min\nolimits_{\R^n}u\in\R.
$$

\begin{lemma}[\!\!\cite{ColesantiLudwigMussnig}, Lemma 5]
\label{le:lk_conv}
Let $u_k,u\in\CV$. If $u_k\eto u_k$, then $\{u_k\leq t\} {\to} \{u\leq t\}$ for every $t\in\R$ with $t\neq \min_{x\in\Rn} u(x)$.
\end{lemma}

\begin{lemma}[\!\!\cite{RockafellarWets}, Proposition 7.2] \label{th:epi_lvl_sets} 
Let $u_k, u \in\CV$. If for each $t\in\R$ there exists a sequence $t_k$ of reals convergent to $t$ with $\{u_k\leq t_k\} \to \{u\leq t\}$, then $u_k \eto u$.  \end{lemma}

We also require the so-called cone property and  uniform cone property for functions and sequences of functions from $\CV$.

\begin{lemma}[\!\!\cite{ColesantiFragala}, Lemma 2.5]
\label{le:cone}
For $u\in\CV$ there exist constants $a,b \in \R$ with $a >0$ such that
\begin{equation*}
u(x)>a|x|+b
\end{equation*}
for every $x\in\Rn$.
\end{lemma}

\begin{lemma}[\!\!\cite{ColesantiLudwigMussnig}, Lemma 8]
\label{le:un_cone}
Let $u_k, u \in\CV$. If $u_k \eto u$, then there exist constants $a,b \in \R$ with $a >0$ such that
\begin{equation*}
u_k(x)>a\,\vert x\vert +b\,\, \text{ and }\,\ u(x)>a\,|x|+b 
\end{equation*}
for every $k\in\N$ and  $x\in\Rn$.
\end{lemma}

Next, we recall some results on valuations on $\CV$. 
For $K\in\Ko$, we define the convex function $\ell_K:\R^n\to [0,\infty]$ by
\begin{equation}\label{ellK}
\epi \ell_K = \pos (K\times\{1\}),
\end{equation}
where $\pos$ stands for positive hull, that is, $\pos(L)=\{t\,z\in\R^{n+1}: z\in L, t\ge 0\}$ for $L\subset \R^{n+1}$. This means that the epigraph of $\ell_K$ is a cone with apex at the origin and $\{\ell_K\leq t \}=t \, K$ for all $t \geq 0$. It is easy to see that $\ell_K$ is an element of $\CV$ for  $K\in\Ko$. Also the (convex) indicator function $\Ind_K$ for $K\in\cK^n$ belongs to $\CV$, where $\Ind_K(x)=0$ for $x\in K$ and $\Ind_K(x)= +\infty$ for $x\not\in K$.

\begin{lemma}[\!\!\cite{ColesantiLudwigMussnig}, Lemma 20]
\label{import_grwth}
For $k\ge 1$, let $\oY:\C{k}\to\R$ be a continuous, translation invariant valuation and let $\psi\in C(\R)$. If
\begin{equation}\label{ind_simple}
\oY(\ell_P+t) = \psi(t) V_k(P)
\end{equation}
for every $P\in\cP_0^k$ and $t\in\R$, then
\begin{equation*}
\oY(\Ind_{[0,1]^k}+t) = \frac{(-1)^k}{k!} \frac{\d^k}{\d t^k} \psi(t)
\end{equation*}
for every $t\in\R$. In particular, $\psi$ is $k$-times differentiable.
\end{lemma}

\begin{lemma}[\!\!\cite{ColesantiLudwigMussnig}, Lemma 23]
\label{le:derivative_has_finite_moment}
Let $\zeta\in C(\R)$  have constant sign on $[t_0,\infty)$ for some $t_0\in\R$. If  there exist $k\in\N$, $c_k\in\R$ and $\psi\in C^k(\R)$  with $\lim_{t\to+\infty} \psi(t)=0$ such that 
$$\zeta(t) = c_k \,\frac{\d^k}{\d t^k}\psi(t)$$ 
for $t\ge t_0$,  then
\begin{equation*}
\Big| \int_{0}^{+\infty} t^{k-1} \zeta(t) \d t\Big| < +\infty.
\end{equation*}
\end{lemma}

The next result, which is based on \cite{Ludwig:sobval}, shows that in order to classify valuations on $\CV$, it is enough to know the behavior of valuations on certain functions.

\begin{lemma}[\!\!\cite{ColesantiLudwigMussnig}, Lemma 17]
\label{le:reduction}
Let $\langle A,+\rangle$ be a topological abelian semigroup with cancellation law and let $\,\oZ_1, \oZ_2:\CV\to \langle A,+\rangle$ be continuous, translation invariant valuations. If $\,\oZ_1(\ell_P+t)=\oZ_2(\ell_P+t)$ for every $P\in\cP_0^n$ and $t\in\R$, then $\oZ_1 \equiv \oZ_2$ on $\CV$.
\end{lemma}

\goodbreak
A function $f:\R^n\to \R$ is \emph{quasi-concave} if its superlevel sets 
$\{f\ge t\}$
are convex for every $t\in\R$. Let $\QC$ denote the space of quasi-concave functions $f: \R^n \to [0, +\infty]$ which are not identically zero, upper semicontinuous and such that
\begin{equation*}
\lim_{\lvert x\rvert \to+ \infty} f(x)=0.
\end{equation*}
Note that $\zeta\circ u\in\QC$ for $\zeta\in\D{k}$ with $k\ge 0$ and $u\in\CV$. 
A natural extension of the volume in $\Rn$ is the integral with respect to the Lebesgue measure, that is, for $f\in\QC$, we set
\begin{equation}\label{eq:int_vol_func}
V_n(f)=\int_{\Rn} f(x) \d x.
\end{equation}
See  \cite{BobkovColesantiFragala} for more information.

\goodbreak

Following \cite{BobkovColesantiFragala}, for $f\in\QC$ and a linear subspace $E\subset \Rn$, we define the \emph{projection function} $ \proj_E f:E\to [0, +\infty]$  for $x\in E$ by
\begin{equation}
\label{eq:def_proj}
\proj_E f(x) = \max_{y\in E^\bot} f(x+y),
\end{equation}
where $E^\bot$ is the orthogonal complement of $E$.
For $t \geq 0$, we have
$\max_{y\in E^\bot} f(x+y)\geq t$ if and only if there exists $y \in E^\bot$ such that $f(x+y)\geq t$. Hence, for $t\ge 0$,
\begin{equation}
\label{eq:proj_func}
\{\proj_E f \geq t \} = \proj_E \{ f\geq t\},
\end{equation}
where $\proj_E$ on the right side denotes the usual projection onto $E$ in $\Rn$.

\section{Valuations on Convex Bodies}

We collect results on valuations on convex bodies and prove two auxiliary results.

\subsection{$\slnbf$ contravariant Minkowski Valuations on Convex Bodies}

For $z\in\sn$, let $z^\bot$ be the subspace orthogonal to $z$. 
The \textit{projection body}, $\oP K$, of the convex body $K\in\Kn$ is defined by
\begin{equation}\label{proj_def}
h(\oP K, z) = V_{n-1}(\proj_{z^{\bot}} K)= \tfrac 12 \int_{\sn} \vert y\cdot z\vert  \d S(K,y)
\end{equation}
for $z\in\sn$. 

More generally, for  a finite  Borel measure $Y$ on $\sn$, we define its \emph{cosine transform} $\Cos{Y}:\Rn\to\R$ by 
\begin{equation*}\label{le:cosine_transf}
\Cos{Y}(z)= \int_{\sn} |y\cdot z| \d Y(y)
\end{equation*}
for $z\in\Rn$. Since $z\mapsto \Cos{Y}(z)$ is easily seen to be  sublinear and non-negative on $\Rn$,  the cosine transform $\Cos{Y}$ is the support function of a convex body that contains the origin. 

\goodbreak

The projection body has useful properties concerning $\sln$ transforms and translations. For  $\phi \in \sln$ and any translation $\tau$ on $\Rn$, we have
\begin{equation}
\oP(\phi K) = \phi^{-t} \oP K \quad\text{ and }\quad \oP(\tau K) = \oP K
\label{eq:prop_pi}
\end{equation}
for all $K\in \Kn$. Moreover, the operator $K\mapsto \oP K$ is continuous and the origin is an interior point of $\oP K$, if $K$ is $n$-dimensional. See \cite[Section 10.9]{Schneider:CB2} for more information on projection bodies.

\goodbreak
We require the following result where the support function of certain  projection bodies is calculated for specific vectors. Let $n\geq 2$. 

\begin{lemma}
\label{le:pi_p_q}
For the polytopes $P=\conv\{0,\tfrac1{2}(e_1+e_2),e_2,\ldots,e_n\}$ and $Q=\conv\{0,e_2,\ldots,e_n\}$ we have
\begin{equation*}
\begin{aligned}
&h(\oP P,e_1)=\tfrac{1}{(n-1)!}		& &h(\oP Q,e_1)=\tfrac{1}{(n-1)!}\\
&h(\oP P,e_2)=\tfrac{1}{2(n-1)!}	& &h(\oP Q,e_2)=0\\
&h(\oP P,e_1+e_2)=\tfrac{1}{(n-1)!}	\qquad & &h(\oP Q,e_1+e_2)=\tfrac{1}{(n-1)!}.
\end{aligned}
\end{equation*}
\end{lemma}

\begin{proof}
We use induction on the dimension and start with $n=2$. In this case, $P$ is a triangle in the plane with vertices $0,\tfrac{1}{2}(e_1+e_2)$ and $e_2$ and $Q$ is just the line segment connecting the origin with $e_2$. It is easy to see that $h(\oP P,e_2)=V_1(\proj_{e_2^\bot} P)=\tfrac{1}{2}$ and $h(\oP Q,e_2)=0$ while $h(\oP P,e_1)=h(\oP Q,e_1)=1$. It is also easy to see that
\begin{equation*}
h(\oP P,e_1+e_2)=h(\oP Q,e_1+e_2) = \sqrt{2} \tfrac{\sqrt{2}}{2} = 1.
\end{equation*}
Assume now that the statement holds for $(n-1)$. All the projections to be considered are simplices that are the convex hull of $e_n$ and a base in $e_n^\perp$ which is just the projection as in the $(n-1)$-dimensional case. Therefore, the corresponding $(n-1)$-dimensional volumes are just $\tfrac{1}{n-1}$ multiplied with the $(n-2)$-dimensional volumes from the previous case. To illustrate this, we will calculate $h(\oP P,e_1+e_2)$ and remark that the other cases are similar. Note that $\proj_{(e_1+e_2)^\bot} P=\conv\{e_n,\proj_{(e_1+e_2)^\bot} P^{(n-1)}\}$, where $P^{(n-1)}$ is the set in $\R^{n-1}$ from the $(n-1)$-dimensional case   embedded via the identification of $\R^{n-1}$ and $e_n^\perp\subset \Rn$. Using the induction hypothesis and  $|e_1+e_2|=\sqrt{2}$, we obtain
\begin{equation*}
V_{n-1}(\proj_{(e_1+e_2)^\bot} P) = \tfrac{1}{n-1}\, V_{n-2}(\proj_{(e_1+e_2)^\bot} P^{(n-1)}) =
\tfrac{1}{\sqrt{2}(n-1)!},
\end{equation*}
and therefore $h(\oP P,e_1+e_2)=\tfrac{1}{(n-1)!}$.
\end{proof}

The first classification of Minkowski valuations was established in \cite{Ludwig:projection}, where the projection body operator was characterized as an $\sln$ contravariant and translation invariant valuation.  The following strengthened version of results from \cite{Ludwig:Minkowski} is due to Haberl. Let $n\ge 3$.

\begin{theorem}[\!\! \cite{Haberl_sln}, Theorem 4]
\label{thm:haberl}
An operator $\oZ:\Ko \to \Kn$ is a continuous, $\sln$~contra\-variant Minkowski valuation if and only if there exists  $c\ge 0$ such that
\begin{equation*}
\oZ K = c \oP K
\end{equation*}
for every $K\in\Ko$.
\end{theorem}

\noindent
For further results on $\sln$ contravariant Minkowski valuations, see  \cite{LiLeng_polytopes,Ludwig:convex, Schuster:Wannerer}.

\goodbreak

\subsection{$\slnbf$ Covariant Minkowski Valuations on Convex Bodies}
The difference body $\oD K$ of a convex body $K\in\Kn$ is defined by $\oD K = K+(-K)$, that is,
\begin{equation*}
h(\oD K,z)= h(K,z)+h(-K,z)=V_{1}(\proj_{E(z)} K)
\end{equation*}
for every $z\in\sn$, where $E(z)$ is the span of $z$.
The moment body $\oM K$ of $K$ is defined by
\begin{equation*}
h(\oM K,z) = \int_{K} |x\cdot z| \d x
\end{equation*}
for every $z\in\sn$. The moment vector $\vm(K)$ of $K$ is defined by
\begin{equation*}
\vm(K) = \int_{K} x \d x
\end{equation*}
and  is an element of $\,\Rn$.

\goodbreak
We require the following result where the support function of certain moment bodies and moment vectors is calculated for specific vectors. Let $n\geq 2$. 

\begin{lemma}
\label{le:t_l_h}
For $s>0$ and $T_s=\conv\{0,s\,e_1,e_2, \ldots, e_n\}$, 
\begin{equation*}
\begin{aligned}
&h(T_s,e_1)= s	& &h(-T_s,e_1)=0\\
&h(\vm(T_s),e_1)=\tfrac{s^2}{(n+1)!}	\qquad & &h(\oM T_s,e_1)=\tfrac{s^2}{(n+1)!}.
\end{aligned}
\end{equation*}

\end{lemma}
\begin{proof}
It is easy to see that $h(T_s,e_1)= s$ and $h(-T_s,e_1)=0$. Let $\phi_s\in\gln$ be such that $e_1\mapsto s\,e_1$ and $e_i\mapsto e_i$ for $i=2, \dots, n$.
Then $T_s = \phi_s T^n$, where $T^n=\conv\{0,e_1,\ldots,e_n\}$ is the standard simplex. Hence,
$$h(\vm(T_s),e_1) = h(\vm(\phi_s T^n), e_1) = |\det \phi_s|\, h(\vm(T^n),(\phi_s)^t e_1)\\= s^2\, h(\vm(T^n),e_1) =  \tfrac{s^2}{(n+1)!},
$$
where $\det$ stands for determinant. Finally, since $e_1\cdot x\geq 0$ for every $x\in T_s$, we have $h(\oM T_s,e_1)=h(\vm(T_s),e_1)$.
\end{proof}

A first classification of $\sln$ covariant Minkowski valuations was established in \cite{Ludwig:Minkowski}, where also the difference body operator was characterized. The following result is due to Haberl. Let $n\ge 3$.

\begin{theorem}[\!\!\cite{Haberl_sln}, Theorem 6]
\label{thm:haberl_covariant}
An operator $\oZ:\Ko \to \Kn$ is a continuous, $\sln$~covariant Minkowski valuation if and only if there exist  $c_1,c_2,c_3\geq 0$ and $c_4\in\R$ such that
\begin{equation*}
\oZ K = c_1\, K + c_2 (-K) + c_3 \oM K + c_4\vm(K)
\end{equation*}
for every $K\in\Ko$.
\end{theorem}

We also require the following result which holds for $n\ge 2$.

\begin{theorem}[\!\!\cite{Ludwig:Minkowski}, Corollary 1.2]
\label{thm:ludwig_covariant}
An operator  $\oZ:\Pn\to \Kn$ is an $\,\sln$ covariant and  translation invariant Minkowski valuation if and only if there exists  $c\geq 0$ such that
\begin{equation*}
\oZ P=c \oD P
\end{equation*}
for every $P\in\Pn$.
\end{theorem}

\noindent
For further results on $\sln$ covariant Minkowski valuations, see  \cite{LiLeng_polytopes,Ludwig:convex, Wannerer2011}.

\subsection{Measure-valued Valuations on Convex Bodies}

Denote by $\meas$ the space of finite Borel measures on $\sn$. Following \cite{Haberl:Parapatits_crelle}, for $p\in\R$, we say that a valuation $\oY:\Po\to\meas$  is $\sln$ \textit{contravariant of degree $p$} if
\begin{equation}\label{measure_contra}
\int_{\sn}b(z) \d\oY(\phi P,z)= \int_{\sn} b( \phi^{-t} z) \d\oY(P,z)
\end{equation}
for every map $\phi\in\sln$, every $P\in\Po$  and every continuous $p$-homogeneous function $b:\Rn\backslash\{0\}\to\R$.

\goodbreak
The following result is due to Haberl and Parapatits. Let $n\ge 3$.

\begin{theorem}[\!\!\cite{Haberl:Parapatits_crelle}, Theorem 1]
\label{thm:sam_haberl_parapatits}
A map $\,\oY:\Po\to\meas$ is a weakly continuous valuation that is $\sln$~contravariant of degree $1$  if and only if there exist  $c_1, c_2\ge0$ such that
\begin{equation*}
\oY(P,\cdot)=c_1 S(P,\cdot)+c_2 S(-P,\cdot)
\end{equation*}
for every $P\in\Po$.
\end{theorem}

Denote by $\mease$ the set of finite \textit{even} Borel measures on $\sn$, that is, measures $Y\in\meas$ with
$Y(\omega)=Y(-\omega)$
for every Borel set $\omega\subset \sn$. We remark that if in the above theorem we also require the measure $\oY(P,\cdot)$ to be even and hence $\oY: \Po\to \mease$, then
there is a constant $c\ge 0$ such
\begin{equation}\label{thm:sam_haberl_parapatits_even}
\oY(P,\cdot)=c\big( S(P,\cdot)+ S(-P,\cdot)\big)
\end{equation}
for every $P\in\Po$.

\section{Measure-valued Valuations on $\CVb$}\label{measdef}

In this section, we extend the \emph{LYZ measure}, that is, the surface area measure of the image of the LYZ~operator, to functions $\zeta\circ u$, where $\zeta\in\Dn$ and $u\in\CV$. 
First, we recall the definition of the LYZ operator on $\sW$ by Lutwak, Yang and Zhang \cite{LYZ2006}.  

\goodbreak
Following  \cite{LYZ2006}, for $f\in \sW$ not vanishing a.e., we define the even Borel measure $S(\os{f}, \cdot)$ on $\sn$ (using the Riesz-Markov-Kakutani representation theorem) by the condition that
\begin{equation}\label{LYZ}
\int_{\sn} b(z) \d S(\os{f},z)=\int_{\Rn} b( \nabla f(x)) \d x
\end{equation}
for every $b:\Rn\to\R$ that is even, continuous and $1$-homogeneous. Since the LYZ measure $S(\os{f}, \cdot)$ is even and not concentrated on a great subsphere of $\sn$ (see \cite{LYZ2006}), the solution to the Minkowski problem implies that there is a unique origin-symmetric convex body $\os{f}$ whose surface area measure is $S(\os{f}, \cdot)$.

If, in addition, $f=\zeta\circ u\in C^{\infty}(\R^n)$ with $\zeta \in\Dn$ and $u\in\CV$, the set $\{f\ge t\}$ is a convex body for $0<t\le \max_{x\in\R^n} f(x)$, since the level sets of $u$ are convex bodies and $\zeta$ is non-increasing with $\lim_{s\to+\infty} \zeta(s)=0$. Hence we may rewrite (\ref{LYZ})  as
\begin{equation}\label{coLYZ}
\int_{\sn} b(z) \d S(\os{f},z)=\int_0^{+\infty} \int_{\sn} b( z) \,dS(\{f\ge t\}, z) \d t.
\end{equation}
Indeed, using that $b$ is 1-homogeneous,  the co-area formula (see, for example, \cite[Section~2.12]{AmbrosioFuscoPallara}),  Sard's theorem, and   the definition of  surface area measure, we obtain
\begin{eqnarray*}
\int_{\Rn} b( \nabla f(x)) \d x &= & \int_{\Rn\cap\{\nabla f \ne 0\}} b\big(\tfrac{\nabla f(x)}{\vert \nabla f(x)\vert}\big) \,\vert \nabla f(x)\vert \d x\\
&=& \int_0^{+\infty} \int_{\partial \{f\ge t\}} b\big(\tfrac{\nabla f(y)}{\vert \nabla f(y)\vert}\big)\d \cH^{n-1}(y) \d t\\
&=& \int_0^{+\infty} \int_{\sn} b(z) \d S(\{f\ge t\},z) \d t, \\
\end{eqnarray*}
where $\H^{n-1}$ denotes the $(n-1)$-dimensional Hausdorff measure.

Formula \eqnref{coLYZ} provides the motivation of our extension of the LYZ operator, for which we require the following result.

\begin{lemma}
\label{le:is_in_bv}
If $\zeta\in \Dn$, then 
\begin{equation*}
\int_{0}^{+\infty} \H^{n-1} (\partial\{\zeta \circ u \geq t\}) \d t < +\infty
\end{equation*}
for every $u\in\CV$.
\end{lemma}
\begin{proof}
Fix $\e>0$ and $u\in\CV$.  Let $\rho_{\e}\in C^{+\infty}(\R)$ denote a standard mollifying kernel such that $\int_{\Rn} \rho_{\e} \d x=1$ and $\rho_{\e}(x)\geq 0$ for all $x\in\Rn$ while the support of $\rho_{\e}$ is contained in a centered ball of radius $\e$. Write $\tau_{\e}$ for the translation $t\mapsto t+\e$ on $\R$ and define $\zeta_\varepsilon(t)$ for $t\in\R$ by
\begin{equation*}
\zeta_\varepsilon(t) = (\rho_{\e} \star (\zeta \circ \tau_{\e}^{-1}))(t) +e^{-t} = \int_{-\e}^{+\e} \zeta(t-\e-s)\rho_{\e}(s) \d s +e^{-t}.
\end{equation*}
It is easy to see, that $\zeta_\varepsilon$ is non-negative and smooth. Since $t\mapsto \int_{-\e}^{+\e} \zeta(t-\e-s)\rho_{\e}(s) \d s$ is decreasing,  $\zeta_\varepsilon$ is strictly decreasing. Since
\begin{equation*}
\int_{-\e}^{+\e} \zeta(t-\e-s)\rho_{\e}(s) \d s \geq \int_{-\e}^{+\e} \zeta(t)\rho_{\e}(s) \d s = \zeta(t),
\end{equation*}
we get $\zeta_\varepsilon(t)\geq \zeta(t)$ for every $t\in\R$. Finally, $\zeta_\varepsilon$ has finite $(n-2)$-nd moment, since $t\mapsto e^{-t}$ has finite $(n-2)$-nd moment and
\begin{eqnarray*}
\int_{0}^{+\infty} t^{n-2} \int_{-\e}^{+\e}  \zeta(t-\e-s)\rho_{\e}(s) \d s \d t &=& \int_{-\e}^{+\e} \rho_{\e}(s) \int_{0}^{+\infty} t^{n-2} \zeta(t-\e-s) \d t \d s\\
&\leq& \int_{-\e}^{+\e} \rho_{\e}(s) \d s \int_{0}^{+\infty} t^{n-2} \zeta(t-2\e) \d t < +\infty.
\end{eqnarray*}
Since $\zeta_\varepsilon\geq \zeta$, we have $\{\zeta \circ u \geq t\} \subseteq \{\zeta_\varepsilon\circ u \geq t\}$ for every $t\in\R$. Since those are compact convex sets for every $t> 0$, we obtain $\H^{n-1}(\partial \{\zeta\circ u\geq t\}) \leq \H^{n-1}(\partial \{\zeta_\varepsilon\circ u \geq t\})$ for every $t>0$. Hence, it is enough to show that
\begin{equation*}
\int_{0}^{+\infty } \H^{n-1}(\partial\{\zeta_\varepsilon\circ u\geq t\}) \d t < +\infty.
\end{equation*}
By Lemma \ref{le:cone}, there exist constants $a,b\in\R$ with $a>0$ such that $u(x)>v(x)=a|x|+b$ for all $x\in\Rn$. Therefore $\zeta_\varepsilon\circ u < \zeta_\varepsilon\circ v $, which implies that  $\{\zeta_\varepsilon\circ u \geq t\} \subset \{\zeta_\varepsilon\circ v \geq t\}$ for every $t> 0$. Hence, by convexity, the substitution $t = \zeta_\varepsilon(s)$ and integration by parts, we obtain
\begin{eqnarray*}
\int_{0}^{+\infty} \H^{n-1} (\partial\{\zeta_\varepsilon \circ u \geq t\}) \d t &<& \int_{0}^{+\infty} \H^{n-1} (\partial\{\zeta_\varepsilon \circ v \geq t\}) \d t\\[-2pt]
&=& \tfrac{n\,v_n}{a^{n-1}} \int_{0}^{\zeta_\varepsilon(b)} ({\zeta_\varepsilon^{-1}(t)-b})^{n-1} \d t\\[-2pt]
&=& -\tfrac{n\,v_n}{a^{n-1}}  \int_{b}^{+\infty}  \underbrace{({s-b})^{n-1} \zeta_\varepsilon'(s)}_{<0} \d s \\[-8pt]
&\le&-\tfrac{n\,v_n}{a^{n-1}}  \underbrace{\liminf_{s\to+\infty} ({s-b})^{n-1}\zeta_\varepsilon(s)}_{\in [0,+\infty]} + \tfrac{n(n-1)\,v_n}{a^{n-1}}  \underbrace{\int_{b}^{+\infty} ({s-b})^{n-2} \zeta_\varepsilon(s) \d s}_{<+\infty}\\[-9pt]
&<& +\infty,
\end{eqnarray*}
where $v_n$ is the volume of the $n$-dimensional unit ball.
\end{proof}

\goodbreak

The previous lemma admits a reverse statement. Let $\zeta\in C(\R)$ be non-negative and decreasing, and assume that 
\begin{equation}\label{new 1}
\int_{0}^{+\infty} \H^{n-1} (\partial\{\zeta \circ u \geq t\}) \d t < +\infty
\end{equation}
for every $u\in\CV$. Then necessarily
\begin{equation}\label{new 2}
\int_0^{+\infty} t^{n-2}\zeta(t) \d t<+\infty,
\end{equation}
i.e. $\zeta\in D^{n-2}(\R)$. Indeed, the following identity holds
\begin{equation}\label{new 3}
\int_0^{+\infty}\H^{n-1}(\partial\{x\colon \zeta(|x|)\ge t\}) \d t= (n-1)
\H^{n-1}({\mathbb S}^{n-1})\
\int_0^{+\infty} t^{n-2}\zeta(t) \d t.
\end{equation}
Therefore, substituting $u(x)=|x|$ in \eqref{new 1} we immediately get \eqref{new 2}. Identity \eqref{new 3} can be easily proved by the co-area formula, when
$\zeta$ is smooth, strictly decreasing and it vanishes in $[t_0,+\infty)$, for some $t_0>0$. The general case is the obtained by a standard approximation argument.

\goodbreak

\begin{lemma}[\bf and Definition]
\label{le:s_lyz_f}
For $u\in\CV$ and $\zeta\in\Dn$, an even finite Borel measure $S(\os{\zeta\circ u},\cdot)$ on $\,\sn$ is defined by the condition that
\begin{equation}
\label{eq:lyz_alt}
\int_{\sn} b(z) \d S(\os{\zeta\circ u},z) = \int_0^{+\infty} \int_{\sn} b(z) \d S(\{\zeta\circ u \geq t\},z)\d t
\end{equation}
for every even continuous function $b:\sn\to\R$. Moreover, if $u_k, u\in\CV$ are such that $u_k \eto u$, then
the measures $S(\os{\zeta \circ u_k},\cdot)$ converge weakly to $S(\os{\zeta\circ u},\cdot)$.
\end{lemma}
\begin{proof}
For fixed  $u\in\CV$ and $\zeta\in\Dn$, we have
\begin{equation*}
\left| \int_0^{+\infty} \int_{\sn} c(z) \d S(\{\zeta\circ u\geq t\},z) \d t \right| \leq \max_{z\in\sn} |c(z)| \int_0^{+\infty}  \H^{n-1}(\partial \{\zeta\circ u \geq t\}) \d t
\end{equation*}
for every continuous function $c:\sn\to\R$. Hence Lemma \ref{le:is_in_bv} shows that
\begin{equation*}
c\mapsto \int_0^{+\infty} \int_{\sn} c(z) \d S(\{\zeta\circ u\geq t\},z)\d t
\end{equation*}
defines a non-negative, bounded linear functional on the space of continuous functions on $\sn$. It follows from the Riesz-Markov-Kakutani representation theorem (see, for example,  \cite{Rudin:RCA}), that there exists a unique Borel measure $\oY(\zeta\circ u,\cdot)$ on $\sn$ such that 
\begin{equation*}
\int_{\sn} c(z) \d \oY(\zeta\circ u,z) = \int_0^{+\infty} \int_{\sn} c(z) \d S(\{\zeta\circ u \geq t\},z)\d t
\end{equation*}
for every continuous function $c:\sn\to\R$. 
Moreover, the measure is finite. For $u\in\CV$ and $\zeta\in\Dn$, define the even Borel measure $S(\os{\zeta\circ u}, \cdot)$ on $\sn$ as
$$S(\os{\zeta\circ u}, \cdot)=\tfrac12 \big( \oY(\zeta\circ u,\cdot) + \oY(\zeta\circ \um,\cdot)\big),$$
where $\um(x)=u(-x)$ for $x\in\Rn$. 
Note that (\ref{eq:lyz_alt}) holds and that $S(\os{\zeta\circ u}, \cdot)$ is the unique even measure with this property.
\goodbreak

Next,  let $u_k,u\in\CV$ with $u_k \eto u$. Fix an even continuous function $b:\sn\to\R$.  
By Lemma \ref{le:lk_conv}, the convex sets $\{u_k\leq t\}$ converge in the Hausdorff metric to $\{u\leq t\}$ for every $t\neq \min_{x\in\Rn} u(x)$, which implies the convergence of $\{\zeta\circ u_k\geq t\}\to\{\zeta\circ u\geq t\}$ for every $t\neq \max_{x\in\Rn}\zeta(u(x))$. Since the map $K\mapsto S(K,\cdot)$ is weakly continuous on the space of convex bodies, we obtain
\begin{equation*}
\int_{\sn} b(z) \d S(\{\zeta\circ u_k \geq t\},z) \to \int_{\sn} b(z) \d S(\{\zeta\circ u \geq t\},z),
\end{equation*}
for a.e.\ $t\geq 0$. By Lemma \ref{le:un_cone}, there exist  $a,d\in\R$ with $a>0$ such that $u_k(x)> v(x)=a|x|+d$ and therefore $\zeta\circ u_k(x) < \zeta\circ v(x)$ for  $x\in\Rn$ and $k\in\N$. By convexity, 
$$\H^{n-1}(\partial \{\zeta\circ u_k\geq t\}) < \H^{n-1}(\partial \{\zeta\circ v\geq t\})$$ 
for every $k\in\N$ and $t>0$ and therefore
\begin{eqnarray*}
\Big| \int_{\sn} b(z) \d S(\{\zeta\circ u_k\geq t\},z)\Big| &\leq& \max_{z\in\sn} \rvert b(z)\lvert\,  \, \H^{n-1}(\partial \{\zeta\circ u_k\geq t\})\\
& <&\max_{z\in\sn} \rvert b(z) \lvert \, \, \H^{n-1}(\partial \{\zeta\circ v\geq t\}).
\end{eqnarray*}
By Lemma \ref{le:is_in_bv}, the function $t\mapsto\int_{\sn} \lvert b(z) \rvert\d S(\{\zeta\circ v\geq t\},z)$ is integrable. Hence, we can apply the dominated convergence theorem to conclude the proof.
\end{proof}

For $p\in\R$, we say that an operator $\oY:\CV\to\meas$ is $\sln$\! \textit{contravariant of degree} $p$ if for $u\in\CV$,
\begin{equation*}
\int_{\sn}b(z) \d\oY(u\circ\phi^{-1},z)= \int_{\sn} b\circ \phi^{-t}(z) \d\oY(u,z)
\end{equation*}
for every $\phi\in\sln$  and every continuous $p$-homogeneous function $b:\Rn\backslash\{0\}\to\R$. This definition generalizes (\ref{measure_contra}) from convex bodies to convex functions. We say that  $\oY$ is \textit{decreasing} on $\CV$, if the real valued function $u\mapsto \oY(u,\sn)$ is decreasing on $\CV$, that is, if $u\ge v$, then $\oY(u,\sn)\le \oY(v,\sn)$. Similarly, we define \emph{increasing} and we say that $\oY$ is \emph{monotone} if it is decreasing or increasing.

\begin{lemma}
\label{le:sam_of_pu_is_val}
For $\zeta\in\Dn$, the map
\begin{equation}
\label{eq:sam_of_pu_is_val}
u\mapsto S(\os{\zeta\circ u},\cdot)
\end{equation}
defines a  weakly continuous,  decreasing  valuation on $\CV$ that is $\sln\!$ contravariant of degree 1 and translation invariant.
\end{lemma}
\begin{proof}
As $K\mapsto S(K,\cdot)$ is translation invariant, it follows from the definition that also $S(\os{\zeta\circ u},\cdot)$ is translation invariant. Lemma \ref{le:s_lyz_f} gives weak continuity. If $u,v\in\CV$ are such that $u\geq v$, then
\begin{equation*}
\{u\leq s\} \subseteq \{v\leq s\},\qquad \{\zeta\circ u\geq t\} \subseteq \{\zeta\circ v \geq t\}
\end{equation*}
and consequently by convexity
\begin{equation*}
S(\{\zeta\circ u \geq t\},\sn)\leq S(\{\zeta\circ v\geq t\},\sn),
\end{equation*}
for all $s\in\R$ and $t> 0$.
For $\phi\in\sln$, 
\begin{equation*}
\{\zeta\circ u \circ \phi^{-1} \geq t\}=\phi\, \{\zeta\circ u\geq t\},
\end{equation*}
and hence by the properties of surface area measure, we obtain
\begin{eqnarray*}
\int_{\sn} b(z) \d S(\os{\zeta\circ u\circ \phi^{-1}},z) &=& \int_0^{+\infty} \int_{\sn} b(z) \d S(\phi \{\zeta\circ u\geq t\},z) \d t\\
&=& \int_0^{+\infty} \int_{\sn} b\circ \phi^{-t}(z) \d S(\{\zeta\circ u\geq t\},z)\d t \\
&=& \int_{\sn} b\circ \phi^{-t} (z) \d S(\os{\zeta\circ u},z)
\end{eqnarray*}
for every continuous 1-homogeneous function $b:\Rn\backslash\{0\}\to \R$. Finally, let $u,v\in\CV$ be such that $u\wedge v\in\CV$. Since $\zeta\in\Dn$ is decreasing, we obtain by (\ref{eq:lvl_sts_val}) and the valuation property of  surface area measure that
\begin{align*}
\int_{\sn} &b(z) \d \big(S(\os{\zeta\circ(u\vee v)},z)+ S(\os{\zeta\circ(u\wedge v)},z)\big)\\
&= \int_0^{+\infty} \int_{\sn} b(z)\d \big(S(\{\zeta\circ u \mn \zeta\circ v \geq t\}, z) + S(\{\zeta\circ u \mx \zeta\circ v \geq t\},z)\big) \d t\\
&= \int_0^{+\infty} \int_{\sn} b(z)  \d \big(S(\{\zeta\circ u\geq t\}\cap\{\zeta\circ v\geq t\},z) + S(\{\zeta\circ u\geq t\}\cup\{\zeta\circ v\geq t\}, z)\big) \d t\\
&= \int_0^{+\infty} \int_{\sn} b(z) \d\big (S(\{\zeta\circ u\geq t\},z) +S(\{\zeta\circ v\geq t\},z)\big) \d t\\
&= \int_{\sn} b(z)  \d\big( S(\os{\zeta\circ u},z)+ \d S(\os{\zeta\circ v},z)\big).
\end{align*}
Hence (\ref{eq:sam_of_pu_is_val}) defines a valuation.
\end{proof}

We remark that Tuo Wang \cite{Tuo_Wang} extended the definition of the  LYZ measure from $\sW$ to the space of functions of bounded variation, $\BV$, using a generalization of (\ref{LYZ}). The co-area formula (see \cite[Theorem~3.40]{AmbrosioFuscoPallara}) and Lemma~\ref{le:is_in_bv} imply that $\zeta \circ u\in \BV$ for every $\zeta\in\Dn$ and $u\in\CV$. However, our approach is slightly different from \cite{Tuo_Wang}.  The extended operators are the same for functions in $\CV$ that do not vanish a.e., but we assign a non-trivial measure also to functions whose support is $(n-1)$-dimensional. In this case, the LYZ measure is concentrated on a great subsphere of $\sn$ and hence we are able to associate to such a function an $(n-1)$-dimensional convex body as solution of the Minkowski problem but not an $n$-dimensional convex body. 
Since Blaschke sums are defined on $n$-dimensional convex bodies,  we do not obtain a characterization of the LYZ operator as a Blaschke valuation on $\CV$. 
Note that Wang's definition allows to extend the LYZ operator to $\BV$ with values in the space of $n$-dimensional convex bodies. However, Wang's extended operators $f\mapsto S(\os{f},\cdot)$ and $f\mapsto \os{f}$ are only  semi-valuations (see \cite{Tuo_Wang_semi} for the definition) but no longer  valuations on $\BV$ and Wang \cite{Tuo_Wang_semi} characterizes  $f\mapsto \os{f}$ as a Blaschke semi-valuation.

\section{$\slnbf$ contravariant Minkowski Valuations on $\CVb$}

The operator that appears in Theorem \ref{contravariant} is defined.  It is shown that it is a continuous, monotone, $\sln$ contra\-variant and  translation invariant Minkowski valuation.

By (\ref{proj_def}) and the definition of the cosine transform, the support function of the classical projection body is the cosine transform of the surface area measure.  Since the measure $S(\os{\zeta\circ u},\cdot)$, defined in Lemma \ref{le:s_lyz_f}, is finite for all $\zeta\in\Dn$ and $u\in\CV$, the cosine transform of $S(\os{\zeta\circ u},\cdot)$ is finite and setting
$$h(\oP\os{\zeta\circ u},z)=\tfrac 12 \Cos{S(\os{\zeta\circ u},\cdot)}(z)$$
for $z\in\R^n$, defines a convex body $\oP\os{\zeta\circ u}$ for $\zeta\in\Dn$ and $u\in\CV$. Here we use that the cosine transform of a measure gives a non-negative and sublinear function, which also shows that $\oP\os{\zeta\circ u}$ contains the origin.
By the definition of the cosine transform and the definition of the LYZ measure $S(\os{\zeta\circ u},\cdot)$, we have
\begin{eqnarray}
\label{eq:pi_lyz_cosine}
h(\oP\os{\zeta\circ u},z)&=& \tfrac 12 \int_{\sn} |y\cdot z |\d S(\os{\zeta\circ u},y)\nonumber\\
&=& \tfrac 12 \int_0^{+\infty} \int_{\sn} |y \cdot z| \d S(\{\zeta\circ u \geq t\},y) \d t \\
&=& \int_{0}^{+\infty} h(\oP \{\zeta\circ u \geq t\},z) \d t\nonumber
\end{eqnarray}
for $\zeta\in\Dn$ and $u\in\CV$. Hence the projection body of $\zeta\circ u$ is a Minkowski average of the classical projection bodies of the sublevel sets of $\zeta\circ u$.

\goodbreak

Using the definition of the classical projection body (\ref{proj_def}), (\ref{eq:proj_func}), the definition (\ref{eq:def_proj})  of projections of quasi-concave functions  and  (\ref{eq:int_vol_func}),  we also  obtain for $z\in\sn$
\begin{equation}\label{geom_co}
\begin{array}{rcl}
h(\oP\os{\zeta\circ u},z)&=&\displaystyle\int_0^{+\infty} h(\oP \{\zeta\circ u\geq t\},z)\d t \\[12pt]
& =& \displaystyle\int_0^{+\infty} V_{n-1} (\proj_{z^\bot} \{\zeta\circ u \geq t\}) \d t \\[12pt]
&=& \displaystyle\int_0^{+\infty} V_{n-1}(\{\proj_{z^\bot} (\zeta\circ u)\geq t\}) \d t \\[16pt]
&=&\displaystyle V_{n-1}(\proj_{z^\bot} (\zeta\circ u)).\\
\end{array}
\end{equation}

\medskip
\noindent
Thus the definition of the projection body of the function $\zeta\circ u$ is analog to the definition of the projection body of a convex body (\ref{proj_def}). In \cite{AMJV}, this connection was established for functions that are log-concave and in $\sW$.

\begin{lemma}
\label{le:z_is_a_val}
For $\zeta\in\Dn$, the map
\begin{equation}\label{map}
u\mapsto \oP\os{\zeta\circ u}
\end{equation}
defines a continuous, decreasing,  $\sln$ contravariant and translation invariant Minkowski valuation on $\CV$.
\end{lemma}

\begin{proof}
Let  $\zeta\in\Dn$ and $u\in\CV$.
By (\ref{eq:prop_pi}) and (\ref{eq:pi_lyz_cosine}), we get for every $\phi\in\sln$ and $z\in\sn$,
\begin{eqnarray*}
h(\oP\os{\zeta\circ u\circ\phi^{-1}}, z)&=&\int_0^\infty h(\oP\{\zeta\circ u\circ \phi^{-1}\ge t\},z) \d t\\
&=&\int_0^\infty h(\oP \phi\{\zeta\circ u\ge t\},z) \d t\\
&=&\int_0^\infty h( \phi^{-t}\oP\{\zeta\circ u\ge t\},z) \d t\\
&=&\int_0^\infty h( \oP\{\zeta\circ u\ge t\}, \phi^{-1} z) \d t\, \,=\, \,h(\oP\os{\zeta\circ u,} \phi^{-1}z).
\end{eqnarray*}
Similarly, we get for every translation $\tau$ on $\Rn$ and $z\in\sn$,
$$h(\oP\os{\zeta\circ u\circ\tau^{-1}}, z)= h(\oP\os{\zeta\circ u}, z).$$
Thus for every $\phi\in\sln$ and every translation $\tau$ on $\Rn$,
\begin{equation*}
\oP\os{\zeta\circ u\circ\phi^{-1}} =\phi^{-t} \oP\os{\zeta\circ u} \quad\text{ and }\quad\oP\os{\zeta\circ u\circ\tau^{-1}}=\oP\os{\zeta\circ u}
\end{equation*}
and 
the map defined in (\ref{map}) is translation invariant and $\sln$ contravariant. 
By Lemma~\ref{le:sam_of_pu_is_val}, the map $u\mapsto S(\os{\zeta\circ u},\cdot)$ is a weakly continuous valuation. Hence, the definition of $\oP\os{\zeta\circ u}$ via the cosine transform and (\ref{eq:spt_fct_addtv}) imply that (\ref{map}) is a continuous Minkowski valuation. Finally, let $\zeta\in\Dn$ and $u,v\in\CV$ be such that $u\geq v$. Then $\{\zeta\circ u\geq t\} \subseteq \{\zeta\circ v\geq t\}$ for every $t \geq 0$ and consequently, $h(\oP \{\zeta\circ u\geq t\},z) \leq h(\oP \{\zeta\circ v\geq t\}, z)$ for every $z\in\sn$ and $t >0$. Hence, for every $z\in\sn$,
\begin{equation*}
h(\oP\os{\zeta\circ u},z) = \int\limits_0^{+\infty} h(\oP \{\zeta\circ u\geq t\},z) \d t \leq \int\limits_0^{+\infty} h(\oP \{\zeta\circ v\geq t\},z) \d t = h(\oP\os{\zeta\circ v},z),
\end{equation*}
or equivalently $\oP\os{\zeta\circ u} \subseteq \oP\os{\zeta\circ v}$. Thus the map defined in (\ref{map}) is decreasing.
\end{proof}

\goodbreak

\section{Classification of $\slnbf$ contravariant Minkowski Valuations}
The aim of this section is to prove Theorem \ref{contravariant}. Let $n\ge 3$ and recall the definition of the cone function $\ell_K$  from \eqnref{ellK}.
\begin{lemma}
\label{le:z_on_lp}
If $\,\oZ:\CV\to \Kn$ is a continuous and $\sln$\! contravariant Minkowski valuation, then there exist continuous functions $\psi,\zeta:\R\to[0,\infty)$ such that
\begin{equation*}
\oZ(\ell_K+t)=\psi(t) \oP K,
\end{equation*}
\begin{equation*}
\oZ(\Ind_K+t)=\zeta(t) \oP K
\end{equation*}
for every $K\in\Ko$ and $t\in\R$.
\end{lemma}
\begin{proof}
For $t\in\R$, define $\oZ_t:\Ko\to\Kn$ as
\begin{equation*}
\oZ_t K = \oZ(\ell_K+t).
\end{equation*}
Now, for $K,L\in\Ko$ such that $K\cup L \in \Ko$, we have $(\ell_K+t) \mn (\ell_L+t) = \ell_{K\cup L}+t$ and $(\ell_K+t) \mx (\ell_L+t) = \ell_{K\cap L}+t$. Using that $\oZ$ is a valuation, we get
\begin{eqnarray*}
\oZ_t K +\oZ_t L &=& \oZ(\ell_K+t) + \oZ(\ell_L+t)\\
&=& \oZ((\ell_K+t) \mx (\ell_L+t)) + \oZ((\ell_K+t) \mn (\ell_L+t)) \\
&=& \oZ_t(K \cup L) + \oZ_t(K \cap L),
\end{eqnarray*}
which shows that $\oZ_t$ is a Minkowski valuation for every $t\in\R$. Since $\oZ$ is $\sln$ contravariant, we obtain for $\phi\in\sln$ that
\begin{equation*}
\oZ_t(\phi K) = \oZ(\ell_{\phi K}+t) = \oZ((\ell_K+t) \circ \phi^{-1}) = \phi^{-t} \oZ(\ell_K+t) = \phi^{-t} \oZ_t K.
\end{equation*}
Therefore, $\oZ_t$ is a continuous, $\sln$ contravariant Minkowski valuation, where the continuity follows from Lemma \ref{le:lk_conv}. By Theorem \ref{thm:haberl}, there exists a non-negative constant $c_t$ such that
\begin{equation*}
\oZ(\ell_K+t) = \oZ_t K = c_t \oP K
\end{equation*}
for all $K\in\Ko$. This defines a function $\psi(t)=c_t$, which is continuous due to the continuity of $\oZ$.
Similarly, using $\oZ_t(K)=\oZ(\Ind_K+t)$, we obtain the function $\zeta$.
\end{proof}

For a continuous, $\sln$ contravariant Minkowski valuation $\oZ:\CV\to \Kn$, we call the function $\psi$ from Lemma \ref{le:z_on_lp} the \textit{cone growth function} of $\oZ$. The function $\zeta$ is called its \textit{indicator growth function}. By Lemma \ref{le:reduction}, we immediately get the following result.

\begin{lemma}
\label{le:z_determined_by_cone_grwth_functions}
Every continuous, $\sln$ contravariant and  translation invariant Minkowski valuation $\oZ:\CV\to\Kn$ is uniquely determined by its cone growth function.
\end{lemma}

Next, we establish an important connection between cone and indicator growth functions.

\begin{lemma}
\label{le:z_grwth_relation}
Let $\oZ:\CV\to\Kn$ be a continuous, $\sln\!$ contravariant and translation invariant Minkowski valuation. The growth functions satisfy
\begin{equation*}
\zeta(t)=\frac{(-1)^{n-1}}{(n-1)!}\frac{\d^{n-1}}{\d t^{n-1}}\psi(t)
\end{equation*}
for every $t\in\R$.
\end{lemma}
\begin{proof}
We fix the $(n-1)$-dimensional linear subspace $E=e_n^\perp$ of $\R^n$. Since $E$ is of dimension $(n-1)$, we can identify the set of functions $u\in\CV$ such that $\dom u\subseteq E$ with $\C{n-1}=\CH$. We define $\oY:\CH\to\R$ by
\begin{equation*}
\oY(u)=h(\oZ(u),e_n).
\end{equation*}
Since $\oZ$ is a Minkowski valuation, $\oY$ is a real valued valuation. Moreover, $\oY$ is continuous and translation invariant, since $\oZ$ has these properties. By the definition of the growth functions we now get
\begin{equation*}
\oY(\ell_P+t)=h(\oZ(\ell_P+t),e_n)=\psi(t)h(\oP P,e_n) = \psi(t) V_{n-1}(P)
\end{equation*}
and
\begin{equation*}
\oY(\Ind_P+t) = h(\oZ(\Ind_P+t),e_n) = \zeta(t)h(\oP P,e_n) = \zeta(t) V_{n-1}(P)
\end{equation*}
for every $P\in\cP_{0}^{n-1}(E)=\{P\in\Po\,:\,P\subset E\}$ and $t\in\R$. Hence, by Lemma \ref{import_grwth},
\begin{equation*}
\zeta(t)=\zeta(t)\,V_{n-1}([0,1]^{n-1}) = \oY(\Ind_{[0,1]^{n-1}}+t) = \frac{(-1)^{n-1}}{(n-1)!}\frac{\d^{n-1}}{\d t^{n-1}}\psi(t)
\end{equation*}
for every $t\in\R$, where $[0,1]^{n-1} = [0,1]^n \cap E$.
\end{proof}

\goodbreak
Next, we establish important properties of the cone growth function.

\begin{lemma}
\label{le:z_grwth_fnc_dcrsng}
If\, $\oZ:\CV\to \Kn$ is a continuous, $\sln$~contravariant and translation invariant Minkowski valuation, then its cone growth function $\psi$  is decreasing and satisfies
\begin{equation}
\label{eq:z_grw_infty}
\lim_{t\to\infty}\psi(t)=0.
\end{equation}
\end{lemma}
\begin{proof}
In order to prove that $\psi$ is decreasing, we have to show that $\psi(s)\geq \psi(t)$ for all $s<t$. Without loss of generality, we  assume that $s=0$, since for arbitrary $s$ we can consider $\widetilde{\oZ}(u)=\oZ(u+s)$ with cone growth function $\widetilde{\psi}$ and $\widetilde{\psi}(0)=\psi(s)$. Hence, for the remainder of the proof we fix an arbitrary $t>0$ and we have to show that $\psi(t)\leq \psi(0)$.

Define $P$ and $Q$ as in Lemma \ref{le:pi_p_q}. Choose $u_t\in\CV$ such that $\epi u_t=\epi \ell_P \cap \{x_1\leq \tfrac{t}{2}\}$. Let $\tau_t$ be the translation $x\mapsto x+\tfrac t2 (e_1+e_2)$ and define $\ell_{P,t}(x)=\ell_P(x)\circ \tau_t^{-1}+t$ and similarly $\ell_{Q,t}(x)=\ell_Q(x)\circ \tau_t^{-1}+t$. Note that
\begin{equation*}
u_t \mn \ell_{P,t} = \ell_P \qquad\text{ and }\qquad u_t \vee \ell_{P,t} = \ell_{Q,t}.
\end{equation*}
Thus, the valuation property of $\oZ$ gives
\begin{equation*}
\oZ(u_t)+\oZ(\ell_{P,t})= \oZ(u_t\wedge \ell_{P,t})+\oZ(u_t\vee \ell_{P,t}) = \oZ(\ell_P)+\oZ(\ell_{Q,t}).
\end{equation*}
Using the translation invariance of $\oZ$ and the definition of the cone growth function, this gives for the support functions
\begin{equation}
h(\oZ(u_t),\cdot)=(\psi(0)-\psi(t))h(\oP P,\cdot)+\psi(t)h(\oP Q,\cdot).
\label{eq:h_u}
\end{equation}
Since $\oZ(u_t)$ is a convex body, its support function is sublinear. This yields
\begin{equation*}
h(\oZ(u_t),e_1+e_2)\leq h(\oZ(u_t),e_1)+h(\oZ(u_t),e_2)
\end{equation*}
and
\begin{multline*}
(\psi(0)-\psi(t))h(\oP P,e_1+e_2)+\psi(t)h(\oP Q,e_1+e_2)\\
\leq (\psi(0)-\psi(t))\big(h(\oP P,e_1)+h(\oP P,e_2)\big) + \psi(t)\big(h(\oP Q,e_1)+h(\oP Q,e_2)\big).
\end{multline*}
Using Lemma \ref{le:pi_p_q}, we obtain
\begin{eqnarray*}
(\psi(0)-\psi(t))\tfrac{1}{(n-1)!}+\psi(t)\tfrac{1}{(n-1)!}
&\leq& (\psi(0)-\psi(t))(\tfrac{1}{(n-1)!}+\tfrac{1}{2(n-1)!}) + \psi(t)(\tfrac{1}{(n-1)!}+0),\\
0 &\leq& (\psi(0) - \psi(t))\tfrac{1}{2(n-1)!},
\end{eqnarray*}
which holds if and only if $\psi(t)\leq \psi(0)$.

\goodbreak
In order to show (\ref{eq:z_grw_infty}), let $t$ in the construction above go to $+\infty$. It is easy to see, that in this case $u_t$ is epi-convergent to $\ell_P$. Since $\psi$ is decreasing and non-negative,  $\lim_{t\to+\infty}\psi(t)=\psi_\infty$ exists. Taking limits in (\ref{eq:h_u}) therefore yields
\begin{equation*}
\psi(0) h(\oP P,\cdot)=h(\oZ(\ell_P),\cdot)=(\psi(0)-\psi_\infty)h(\oP P,\cdot)+\psi_\infty\,h(\oP Q,\cdot).
\end{equation*}
Evaluating at $e_2$ now gives $\psi_\infty=0$.
\end{proof}

By Lemma \ref{le:reduction}, we obtain the following result as an immediate corollary from the last result. We call a Minkowski valuation on $\CV$ \emph{trivial} if $\oZ(u)=\{0\}$ for $u\in\CV$.

\begin{lemma}\label{contra_trivial}
Every  continuous, increasing, $\sln$ contra\-variant and translation invariant Minkowski valuation on $\CV$ is trivial.\end{lemma}

Lemma \ref{le:z_grwth_relation} shows that the indicator growth function $\zeta$ of a continuous, $\sln$ contravariant and trans\-lation invariant Minkowski valuation $\oZ$ determines its cone growth function $\psi$ up to a polynomial of degree less than $n-1$. By Lemma \ref{le:z_grwth_fnc_dcrsng}, $\lim_{t\to\infty} \psi(t)=0$ and hence the polynomial is also  determined by $\zeta$. Thus $\psi$  is completely determined by the indicator growth function of $\oZ$ and  Lemma \ref{le:z_determined_by_cone_grwth_functions} immediately implies the following result.

\begin{lemma}
\label{le:reduction2}
Every continuous, $\sln$ contravariant and translation invariant Minkowski valuation $\oZ:\CV\to\Kn$ is uniquely determined by its indicator growth function.
\end{lemma}

\subsection{Proof of Theorem \ref{contravariant}}

If $\zeta\in\Dn$, then  Lemma \ref{le:z_is_a_val} shows that the operator $u \mapsto \oP\os{\zeta\circ u}$ defines a continuous, decreasing,  $\sln$ contravariant  and translation invariant Minkowski valuation on $\CV$.

Conversely, let a continuous, monotone,  $\sln$ contravariant  and translation invariant Minkowski valuation  $\oZ$ be given and let  $\zeta$ be its indicator growth function. Lemma \ref{contra_trivial} implies that we may assume that $\oZ$ is decreasing. It follows from the definition of $\zeta$ in Lemma \ref{le:z_on_lp} that $\zeta$ is non-negative and continuous. To see that $\zeta$ is decreasing, note that by the definition of $\zeta$ in in Lemma \ref{le:z_on_lp},
\begin{equation*}
h(\oZ(\Ind_{[0,1]^n}+t),e_1)=\zeta(t)\,h(\oP [0,1]^n,e_1)=\zeta(t)
\end{equation*}
for every $t\in\R$ and that $\oZ$ is decreasing. By Lemma \ref{le:z_grwth_relation} combined with Lemma \ref{le:derivative_has_finite_moment}, the function $\zeta$ has finite $(n-2)$-nd moment.  Thus $\zeta\in\Dn$.

For $u=\Ind_P+t$ with  $P\in\Po$ and $t\in\R$,  we obtain by (\ref{eq:pi_lyz_cosine}) that
\begin{equation*}
h(\oP\os{\zeta\circ u},z) = \int_0^{+\infty} h(\oP\{\zeta\circ u\geq s\},z) \d s = \zeta(t)\, h(\oP P,z)
\end{equation*}
for every $z\in\sn$. Hence
$\oP\os{\zeta\circ (\Ind_P+t)}= \zeta(t)\oP P$
for $P\in\Po$ and $t\in\R$. By Lemma \ref{le:z_is_a_val},
$$u \mapsto \oP\os{\zeta\circ u}$$
defines a continuous, decreasing,  $\sln$ contravariant  and translation invariant Minkowski valuation on $\CV$ and $\zeta$ is its indicator growth function. Thus Lemma \ref{le:reduction2} completes the proof of the theorem.

\goodbreak
\section{Classification of Measure-valued  Valuations}\label{measure}

The aim of this section is to prove Theorem \ref{thm:measure}.
Let $n\ge 3$.

\begin{lemma}
\label{le:sam_on_Ind}
If $\,\oY:\CV\to\mease$ is a weakly continuous valuation that is $\sln$ contravariant of degree $1$, then there exist  continuous functions $\psi,\zeta:\R\to[0,\infty)$ such that
\begin{eqnarray*}
\oY(\ell_K+t,\cdot)&=&\tfrac 1 2 \psi(t)\big(S(K,\cdot)+S(-K,\cdot)\big),\\
\oY(\Ind_K+t,\cdot)&=&\tfrac 1 2 \zeta(t)\big(S(K,\cdot)+S(-K,\cdot)\big)
\end{eqnarray*}
for every $K\in\Ko$ and $t\in\R$.
\end{lemma}
\begin{proof}
For $t\in\R$, define $\oY_t:\Ko\to\mease$ as
\begin{equation*}
\oY_t(K,\cdot) =\oY(\ell_K+t,\cdot).
\end{equation*}
As in the proof of Lemma \ref{le:z_on_lp}, we see that $\oY_t$ is a weakly continuous valuation that is $\sln$ contra\-variant  of degree $1$ for every $t\in\R$. By Theorem \ref{thm:sam_haberl_parapatits} and (\ref{thm:sam_haberl_parapatits_even}), for $t\in\R$, there is $c_{t}\ge 0$ such that
\begin{equation*}
\oY_t(K,\cdot)= \oY(\ell_K+t,\cdot) = c_{t}\big( S(K,\cdot) +  S(-K,\cdot)\big)
\end{equation*}
for all $K\in\Ko$. This defines a non-negative function $\psi(t)=\tfrac 12 c_{t}$. Since $t\mapsto \oY(\ell_K+t, \sn)$ is continuous, also  $\psi$ is  continuous. 
The result for indicator functions and $\zeta$ follows along similar lines.
\end{proof}

For a weakly continuous  valuation $\oY:\CV\to\mease$  that is $\sln$ contravariant of degree $1$, we call the function $\psi$  from Lemma \ref{le:sam_on_Ind}, the \textit{cone growth function} of $\oY$ and we call  the function $\zeta$ its \textit{indicator growth function}. 

\begin{lemma}
\label{le:props_via_cosine_transf}
If $\,\oY:\CV\to\mease$ is a weakly continuous valuation that is $\sln$ contravariant of degree $1$ and translation invariant,  then
\begin{equation*}
\zeta(t)= \frac{(-1)^{n-1}}{(n-1)!} \frac{\d^{n-1}}{\d t^{n-1}} \psi(t).
\end{equation*}
Moreover, $\psi$ is decreasing and  $\,\lim_{t\to+\infty} \psi(t)=0$.
\end{lemma}

\begin{proof}
Recall that the cosine transform $\Cos{\oY(u,\cdot)}$ is the support function of a convex body that contains the origin for every $u\in\CV$. By the properties of $\oY$, this induces a continuous, $\sln$ contravariant  and translation invariant Minkowski valuation $\oZ:\CV\to\Kn$ via
\begin{equation*}
h(\oZ(u),y)=\tfrac 12 \Cos{\oY(u,\cdot)}(y)
\end{equation*}
for $y\in\Rn$.
By Lemma \ref{le:sam_on_Ind}, we have
\begin{equation*}
h(\oZ(\ell_K+t),y) = \tfrac12 \Cos{ \big(\tfrac 12 \psi(t)(S(K,\cdot)+S(-K,\cdot))\big)} (y) = \psi(t) h(\Pi K, y)
\end{equation*}
for every $K\in\Ko$, $t\in\R$ and $y\in\Rn$. Hence, by Lemma \ref{le:z_on_lp}, the function $\psi$ is the cone growth function of $\oZ$. Similarly, it can be seen, that $\zeta$ is the indicator growth function of $\oZ$. The result now follows from Lemma \ref{le:z_grwth_relation} and Lemma \ref{le:z_grwth_fnc_dcrsng}.
\end{proof}

\begin{lemma}\label{measure_increasing}
Every weakly continuous,  increasing  valuation $\oY:\CV\to\mease$ that is $\sln$\! contravariant of degree $1$ and translation invariant is trivial.
\end{lemma}
\begin{proof}
Since $\oY$ is increasing,  Lemma \ref{le:sam_on_Ind} implies that for $s<t$
\begin{eqnarray*}
\oY(\ell_K+s,\sn)&\leq& \oY(\ell_K+t,\sn),\\
\psi(s) \big(S(K,\sn)+S(-K,\sn)\big)&\leq& \psi(t) \big( S(K,\sn)+S(-K,\sn)\big)
\end{eqnarray*}
for every $K\in\Ko$. Hence, $\psi$ is an increasing function. By Lemma \ref{le:props_via_cosine_transf}, $\psi\equiv 0$.  Lemma \ref{le:reduction} implies that $\oY$ is trivial.
\end{proof}

\begin{lemma}
\label{le:mu_determined_by_ind_grwth_functions}
Every weakly continuous valuation $\oY:\CV\to\mease$ that is $\sln\!$ contravariant of degree 1 and  translation inva\-riant   is uniquely determined by its indicator growth function.
\end{lemma}

\begin{proof}
By Lemma \ref{le:props_via_cosine_transf}, we have $\lim_{t\to+\infty} \psi(t)=0$ and $\zeta(t)=\frac{(-1)^{n-1}}{(n-1)!}\frac{\d^{n-1}}{\d t^{n-1}}\psi(t)$. This shows that $\zeta$ uniquely determines $\psi$. 
Since Lemma \ref{le:reduction} implies that $\oY$ is determined by its cone growth function, this
implies the statement of the lemma. 
\end{proof}

\subsection{Proof of Theorem \ref{thm:measure}}

By Lemma \ref{le:sam_of_pu_is_val}, the map $\oY:\CV\to \mease$ defined in (\ref{eq:even_sam_val}) is a weakly continuous, decreasing valuation that is $\sln$ contravariant of degree $1$ and translation invariant.

Conversely, let $\oY:\CV\to \mease$ be a weakly continuous,  monotone  valuation that is $\sln$ contravariant of degree $1$ and translation invariant. Let $\zeta:\R\to[0,\infty)$ be its indicator growth function. If $\oY$ is increasing, then Lemma \ref{measure_increasing} shows that $\oY$ is trivial. Hence we may assume that $\oY$ is decreasing.
Lemma \ref{le:props_via_cosine_transf} combined with Lemma \ref{le:derivative_has_finite_moment} implies that $\zeta\in\D{n-2}$.

Now, for $u=\Ind_K+t$ with  $K\in\Ko$ and $t\in\R$  we obtain by Lemma \ref{le:sam_on_Ind} and by the definition of $S(\os{\zeta\circ u}, \cdot)$ in Lemma \ref{le:s_lyz_f} that 
\begin{equation*}
\oY(u,\cdot) = \tfrac 12 \zeta(t) (S(K,\cdot)+S(-K,\cdot))=S(\os{\zeta\circ u},\cdot).
\end{equation*}
By Lemma \ref{le:sam_of_pu_is_val}, 
$$u\mapsto S(\os{\zeta\circ u},\cdot)$$
defines a weakly continuous, decreasing  valuation on $\CV$ that is $\sln$ contravariant of degree $1$ and translation invariant and $\zeta$ is its indicator growth function. Thus Lemma \ref{le:mu_determined_by_ind_grwth_functions} completes the proof of the theorem.

\section{$\slnbf$ covariant Minkowski Valuations on $\CVb$}

The operator that appears in Theorem \ref{covariant} is discussed.  It is shown that it is a continuous, monotone, $\sln$ covariant and  translation invariant Minkowski valuation. Moreover, a geometric interpretation is derived.

\goodbreak
We require the following results.

\begin{lemma}
\label{le:body_finite}
For $\zeta\in\D{0}$,  we have 
$\displaystyle\big\vert \int_0^{+\infty} h(\{\zeta\circ u \geq t\},z)\d t \big\vert<+\infty\,$
for every function $u\in\CV$ and $z\in\sn$.
\end{lemma}
\begin{proof}
Fix  $\e>0$ and $u\in\CV$. Let $\rho_{\e}\in C^{+\infty}(\R)$ denote a standard mollifying kernel such that $\int_{\Rn} \rho_{\e}(x) \d x = 1$, $\supp \rho_{\e} \subseteq B_{\e}(0)$ and $\rho_{\e}(x)\geq 0$ for all $x\in\Rn$. Write $\tau_{\e}$ for the translation $t\mapsto t+\e$ on $\R$ and define $\zeta_\varepsilon(t)$ for $t\in\R$ as 
\begin{equation*}
\zeta_\varepsilon(t)=(\rho_{\e}\star(\zeta\circ\tau_{\e}^{-1}))(t)+e^{-t} = \int_{-\e}^{+\e} \zeta(t-\e-s)\rho_{\e}(s) \d s + e^{-t}.
\end{equation*}
\goodbreak\noindent
As in the proof of Lemma \ref{le:is_in_bv}, it is easy to see that $\zeta_\varepsilon$ is smooth and strictly decreasing and that
\begin{equation*}
\int_0^{+\infty} \zeta_\varepsilon(t) \d t <+\infty.
\end{equation*}
Moreover, $\zeta_\varepsilon(t)> \zeta(t)\geq 0$ for every $t\in\R$. Hence, $\{\zeta\circ u \geq t\} \subseteq  \{\zeta_\varepsilon\circ u \geq t\}$ for every $t \geq 0$ and therefore it suffices to show that
\begin{equation*}
\big\vert \int_0^{+\infty} h(\{\zeta_\varepsilon \circ u \geq t\},z)\d t \big\vert<+\infty
\end{equation*}
for every $z\in\sn$. By Lemma \ref{le:cone}, there exist constants $a,b\in\R$ with $a>0$ such that $u(x)>v(x)=a|x|+b$ for all $x\in\Rn$. Hence, by substituting $t = \zeta_\varepsilon(s)$ and by integration by parts, we obtain
\begin{eqnarray*}
\big \vert\int_0^{+\infty}  h(\{\zeta_\varepsilon\circ u \geq t\},z)  \d t \big\vert &\leq& \int_0^{+\infty} h(\{\zeta_\varepsilon\circ v \geq t\},z) \d t\\
&=& \tfrac 1a \int_0^{\zeta_\varepsilon(b)} ({\zeta_\varepsilon^{-1}(t)-b}) \d t\\[-4pt]
&=&  - \tfrac1a \int_b^{+\infty} \underbrace{({s-b}) \,\zeta_\varepsilon'(s)}_{<0} \d s\\[-8pt]
&\le& - \tfrac 1a \,\underbrace{\liminf_{s\to +\infty} ({s-b}) \,\zeta_\varepsilon(s)}_{\in[0,+\infty]} + \tfrac 1a \underbrace{\int_b^{+\infty} \zeta_\varepsilon(s) \d s}_{<+\infty} <+\infty,
\end{eqnarray*}
\vskip -8pt \noindent
which concludes the proof.
\end{proof}
\goodbreak

\begin{lemma}[\bf and Definition]
\label{le:body_is_a_val}
For $\zeta\in\Dz$, the map $u\mapsto\oid{\zeta\circ u}$ from $\CV$ to $\Kn$,  defined  for $z\in\sn$ by
\begin{equation*}
h(\oid{\zeta\circ u},z)=\int\limits_0^{+\infty} h(\{\zeta\circ u \geq t\},z)\d t,
\end{equation*}
is a continuous, decreasing, $\sln\!$ covariant Minkowski valuation.
\end{lemma}

\begin{proof}
Let $u,v\in\CV$ be such that $u\geq v$. Then
\begin{equation*}
\{\zeta\circ u \geq t\}\subseteq \{\zeta\circ v \geq t\}
\end{equation*}
for every $t\geq 0$ and consequently,
\begin{equation*}
h(\{\zeta\circ u \geq t\},z) \leq h(\{\zeta\circ v \geq t\},z)
\end{equation*}
for every $z\in\sn$. Since the integral in the definition of $\oid{\zeta\circ u}$ converges by Lemma \ref{le:body_finite}, this shows that $u\mapsto \oid{\zeta\circ u}$ is well-defined and decreasing on $\CV$.

\goodbreak
Now, let $u\in\CV$ and $u_k\in\CV$ be such that $\elim_{k\to\infty}u_k=u$. By Lemma~\ref{le:lk_conv}, the sets $\{u_k\leq t\}$ converge in the Hausdorff metric to $\{u\leq t\}$ for every $t\neq \min_{x\in\Rn} u(x)$, which is equivalent to the convergence $\{\zeta\circ u_k \geq t\}\to \{\zeta\circ u \geq t\}$ for every $t\neq \max_{x\in\Rn}\zeta(u(x))$.  By Lemma \ref{le:un_cone}, there exist constants $a,b\in\R$ with $a>0$ such that for every $k\in\N$ and $x\in\Rn$
\begin{equation*}
u_k(x)>v(x)=a|x|+b
\end{equation*}
and therefore $\zeta(u_k(x)) < \zeta(v(x))$ for every $x\in\Rn$ and $k\in\N$ and hence also
\begin{equation*}
\vert h(\{\zeta\circ u_k \geq t\},z) \vert  \leq  h(\{\zeta\circ v \geq t\},z)
\end{equation*}
for every $t\geq 0, k\in\N$ and $z\in\sn$ where we have used the symmetry of $v$. By Lemma \ref{le:body_finite}, we can apply the dominated convergence theorem, which shows that $u\mapsto \oid{\zeta\circ u}$ is continuous.

Finally, since
\begin{equation*}
u\mapsto \{\zeta\circ u \geq t\}
\end{equation*}
defines an $\sln$ covariant Minkowski valuation for every $t> 0$, it is easy to see that also $u\mapsto \oid{\zeta\circ u}$ has these properties.
\end{proof}

Let $f=\zeta\circ u$ with $\zeta\in\Dz$ and $u\in\CV$. Write $E(z)$ for the linear span of $z\in \sn$. 
By the definition of the level set body, the difference body, the projection of a quasi-concave function (\ref{eq:def_proj}), and (\ref{eq:proj_func}), we have
\begin{eqnarray*}
h(\oD \oid{f}, z)&=& h(\oid{f}, z) + h(-\oid{f}, z)\\[3pt]
&=& \int_0^{+\infty} h(\{f\geq t\},z) +h(-\{f\geq t\},z) \d t\\
&=& \int_0^{+\infty} h(\oD\{f\geq t\},z) \d t\\
&=& \int_0^{+\infty} V_1(\proj_{E(z)} \{f\geq t\})\d t\\[3pt]
&=&V_1(\proj_{E(z)} f).
\end{eqnarray*}
This corresponds to the geometric interpretation of the projection body from (\ref{geom_co}).

\goodbreak

\begin{lemma}
\label{le:diff_body_is_a_val}
For $\zeta\in\Dz$,  the map $u\mapsto\oD \oid{\zeta\circ u}$ from $\CV$ to $\Kn$
is a continuous, decreasing,  $\sln$ covariant and translation invariant Minkowski valuation.
\end{lemma}
\begin{proof}
For every translation $\tau$ on $\R^n$ and $u\in\CV$, we have 
$$h(\oD\oid{\zeta\circ u\circ\tau^{-1}}, z) = \int\limits_0^{+\infty} h(\oD\{\zeta\circ u\circ \tau^{-1}\ge t,z\}\d t =  \int\limits_0^{+\infty} h(\oD\{\zeta\circ u\ge t,z\}\d t= h(\oD\oid{\zeta\circ u}, z),$$
since the difference body operator is translation invariant. The further properties follow immediately from the properties of the level set body proved in Lemma \ref{le:body_is_a_val}.
\end{proof}

\goodbreak

\section{Classification of $\sln$ Covariant Minkowski Valuations}

The aim of this section is to prove Theorem \ref{covariant}. Let $n\geq 3$. 

\begin{lemma}
\label{le:covariant_growth_functions}
If $\,\oZ:\CV\to \Kn$ is a continuous, $\sln$ covariant Minkowski valuation, then there exist continuous functions $\psi_1,\psi_2,\psi_3:\R\to[0,\infty)$ and $\psi_4:\R\to\R$ such that
\begin{equation*}
\oZ(\ell_K+t)=\psi_1(t)K+\psi_2(t)(-K)+\psi_3(t)\oM K + \psi_4(t) \vm(K)
\end{equation*}
for every $K\in\Ko$ and $t\in\R$. If $\,\oZ$ is also translation invariant, then there exists a continuous function $\zeta:\R\to[0,\infty)$ such that
\begin{equation*}
\oZ(\Ind_K+t)=\zeta(t) \oD K
\end{equation*}
for every $K\in\Kn$ and $t\in\R$.
\end{lemma}
\begin{proof}
For $t\in\R$, define $\oZ_t:\Ko\to\Kn$ as
$\oZ_tK=\oZ(\ell_K+t)$. 
It is easy to see, that $Z_t$ defines a continuous, $\sln$ covariant Minkowski valuation on $\Ko$ for every $t\in\R$. Therefore, by Theorem \ref{thm:haberl_covariant}, for every $t\in\R$ there exist constants $c_{1,t},c_{2,t},c_{3,t}\geq 0$ and $c_{4,t}\in\R$ such that
\begin{equation*}
\oZ(\ell_K+t)=\oZ_t K=c_{1,t}K+c_{2,t}(-K)+c_{3,t} \oM K + c_{4,t} \vm(K)
\end{equation*}
for every $K\in\Ko$. This defines functions $\psi_i(t)=c_{i,t}$ for $1\leq i \leq 4$. By the continuity of $\oZ$,
\begin{equation*}
t\mapsto h(\oZ(\ell_{T_s}+t),e_1)=s \psi_1(t) + \frac{s^2}{(n+1)!}(\psi_3(t)+\psi_4(t))
\end{equation*}
is continuous for every $s>0$, where $T_s$ is defined as in Lemma \ref{le:t_l_h}. Setting $s=1$ and $s=2$ shows that
\begin{equation*}
t\mapsto \psi_1(t)+\frac{1}{(n+1)!}(\psi_3(t)+\psi_4(t)),
\end{equation*}
\begin{equation*}
t\mapsto 2\psi_1(t)+\frac{4}{(n+1)!}(\psi_3(t)+\psi_4(t))
\end{equation*}
are continuous functions. Hence $\psi_3+\psi_4$ and $\psi_1$ are continuous functions. The continuity of the map $t\mapsto h(\oZ(\ell_{T_s}+t),-e_1)$ shows that $\psi_3-\psi_4$ and $\psi_2$ are continuous. Hence, also $\psi_3$ and $\psi_4$ are continuous functions.

Similarly, if $\oZ$ is also translation invariant, we consider $\oY_t(K)=\oZ(\Ind_K+t)$, which defines a continuous, translation invariant and $\sln$ covariant Minkowski valuation on $\Kn$ for every $t\in\R$. Therefore, by Theorem \ref{thm:ludwig_covariant}, there exists a non-negative constant $d_t$ such that
\begin{equation*}
\oZ(\Ind_K+t)=\oY_t(K)=d_t \oD K
\end{equation*}
for every $t\in\R$ and $K\in\Ko$. This defines a function $\zeta(t)=d_t$, which is continuous due to the continuity of $\oZ$.
\end{proof}

\begin{lemma}
\label{le:covariant_valuations_simple}
If $\,\oZ:\CV\to\Kn$ is a continuous, $\sln$ covariant Minkowski valuation, then, for $e\in\sn$,
\begin{equation*}
h(\oZ(v),e)=0
\end{equation*}
for every $v\in\CV$ such that $\dom v$ lies in an affine subspace orthogonal to $e$. Moreover, if $\vartheta$ is the orthogonal reflection at $e^\bot$, then
\begin{equation*}
h(\oZ(u),e)=h(\oZ(u\circ \vartheta^{-1}),-e)
\end{equation*}
for every $u\in\CV$.
\end{lemma}
\begin{proof}
By Lemma \ref{le:covariant_growth_functions}, we have $h(\oZ(\ell_K),e)=0$ for every $K\in\Ko$ such that $K\subset e^\bot$. Hence,  Lemma \ref{le:reduction} implies that $h(\oZ(v),e)=0$ for every $u\in\CV$ such that $\dom v \subset e^\bot$. By the  translation invariance of $\oZ$, this also holds for $v\in\CV$ whose $\dom v$ lies in an affine subspace orthogonal to $e$.

Similarly, for every $K\in\Ko$, we have $h(K,e)=h(\vartheta K,-e)$ and $h(-K,e)=h(-\vartheta K,-e)$ while  $h(\vm(K),e)=h(\vm(\vartheta K),-e)$ and $h(\oM K,e)=h(\oM(\vartheta K),-e)$. Hence Lemma \ref{le:covariant_growth_functions}implies that $h(\oZ(\ell_K),e)=h(\oZ(\ell_K\circ \vartheta^{-1}),-e)$. The claim follows again from Lemma~\ref{le:reduction}.
\end{proof}

In the proof of the next lemma, we use the following classical result due to H.A.\ Schwarz (cf.\ \cite[p.\ 37]{NatansonII}).
Suppose a real valued function $\psi$ is defined and continuous on the closed interval $I$. If
\begin{equation*}
\lim_{h\to 0} \frac{\psi(t+h)-2\psi(t)+\psi(t-h)}{h^2}=0
\end{equation*}
everywhere in the interior of $I$, then $\psi$ is an affine function.

\begin{lemma}
\label{le:covariant_growth_functions_2}
Let $\oZ:\CV\to\Kn$ be a continuous, $\sln$ covariant and translation invariant Minkowski valuation and let $\psi_1,\psi_2,\psi_3$ and $\psi_4$ be the functions from Lemma \ref{le:covariant_growth_functions}. Then $\psi_1$ and $\psi_2$ are continuously differentiable, $\psi_1'=\psi_2'$ and both $\psi_3$ and $\psi_4$ are constant.
\end{lemma}
\begin{proof}
For  a closed interval $I$ in the span of $e_1$,  let the function $u_I\in\CV$ be defined by
\begin{equation*}
\{u_I <0\}=\emptyset,\quad \{u_I \leq s\} = I + \conv\{0, s\, e_2, \ldots, s \, e_n\}
\end{equation*}
for every $s\geq 0$. By the properties of $\oZ$ it is easy to see that the map $I\mapsto h(\oZ(u_I+t),e_1)$ is a real valued, continuous, translation invariant valuation on $\cK^1$ for every $t\in\R$. Hence, it is easy to see that there exist functions $\zeta_0,\zeta_1:\R\to\R$ such that
\begin{equation}
\label{eq:h_z_u_I}
h(\oZ(u_I+t),e_1)=\zeta_0(t)+\zeta_1(t)V_1(I)
\end{equation}
for every $I\in\cK^1$ and $t\in\R$ (see, for example, \cite[p.\ 39]{Klain:Rota}). Note, that by the continuity of $\oZ$, the functions $\zeta_0$ and $\zeta_1$ are continuous.

For $r,h>0$, let $T_{r/ h} =\conv\{0, \frac r h\,e_1,e_2, \dots, e_n\}$. Define the function $u_r^h$ by
\begin{equation*}
\{u_r^h \leq s\} = \{\ell_{T_{r/h}} \leq s\} \cap \{x_1 \leq r \}
\end{equation*}
for every $s\in\R$. It is easy to see that $u_r^h \in\CV$ and that
\begin{equation*}
\{u_r^h \leq s\} \cup \{\ell_{T_{r/h}}\circ \tau_r^{-1} + h \leq s \} = \{\ell_{T_{r/h}} \leq s\},
\end{equation*}
\begin{equation*}
\{u_r^h \leq s\} \cap \{\ell_{T_{r/h}}\circ \tau_r^{-1} + h \leq s \} \subset \{x_1 = r\}
\end{equation*}
for every $s\in\R$, where $\tau_r$ is the translation $x\mapsto x+ r e_1$. By translation invariance, the valuation property and Lemma \ref{le:covariant_valuations_simple}, this gives
\begin{equation*}
h(\oZ(u_r^h+t),e_1)=h(\oZ(\ell_{T_{r/h}}+t),e_1)-h(\oZ(\ell_{T_{r/h}}+t+h),e_1)
\end{equation*}
for every $t\in\R$. Note, that by Lemma \ref{th:epi_lvl_sets} 
we have $u_r^h \eto u_{[0,r]}$ as $h\to 0$. Hence, using the continuity of $\oZ$, Lemma \ref{le:covariant_growth_functions} and Lemma \ref{le:t_l_h}, we obtain
\begin{multline*}
h(\oZ(u_{[0,r]}+t),e_1) = \lim_{h\to 0^+} h(\oZ(u_r^h+t),e_1)\\
= \lim_{h\to 0^+} \Big(r\, \frac{\psi_1(t)-\psi_1(t+h)}{h} +  \frac{r^2}{(n+1)!}\frac{(\psi_3+\psi_4)(t)-(\psi_3+\psi_4)(t+h)}{h^2} \Big)
\end{multline*}
for every $t\in\R$ and $r>0$. Comparison with (\ref{eq:h_z_u_I}) now gives
\begin{equation}
\label{eq:diff_psi_134}
\zeta_1(t) = \lim_{h\to 0^+} \frac{\psi_1(t)-\psi_1(t+h)}{h},\quad 0 = \lim_{h\to 0^+} \frac{(\psi_3+\psi_4)(t)-(\psi_3+\psi_4)(t+h)}{h^2}.
\end{equation}
Similarly, since also $u_r^h-h \eto u_{[0,r]}$ as $h\to 0$, we obtain
\begin{equation*}
\zeta_1(t) = \lim_{h\to 0^+} \frac{\psi_1(t-h)-\psi_1(t)}{h},\quad 0 = \lim_{h\to 0^+} \frac{(\psi_3+\psi_4)(t-h)-(\psi_3+\psi_4)(t)}{h^2}.
\end{equation*}
Hence, $\psi_1$ is continuously differentiable with $-\psi_1'= \zeta_1$. In addition, by H.A. Schwarz's result, the function $\psi_3+\psi_4$ is linear and hence by (\ref{eq:diff_psi_134}) it must be constant.

Now, let $\vartheta$ denote the reflection at $\{x_1=0\}=e_1^\bot$. Lemma \ref{le:covariant_valuations_simple} and the translation invariance of $\oZ$ give
\begin{align*}
h(\oZ(u_{[0,r]}+t),e_1) &= h(\oZ(u_{[0,r]} \circ \vartheta^{-1} +t), -e_1)\\
&=  h(\oZ(u_{[-r,0]}+t),-e_1) = h(\oZ(u_{[0,r]}+t),-e_1)
\end{align*}
for every $t\in\R$. Repeating the arguments from above, but evaluating at $-e_1$, shows that $-\psi_2'=\zeta_1$ and $\psi_3-\psi_4$ is constant. Hence, both $\psi_3$ and $\psi_4$ are constant.
\end{proof}

\begin{lemma}
\label{le:covariant_growth_functions_3}
If the operator $\,\oZ:\CV\to \Kn$ is a continuous, $\sln$ covariant and translation invariant Minkowski valuation,  then there exists a non-negative $\psi\in C^1(\R)$  such that
\begin{equation*}
\oZ(\ell_K+t)=\psi(t) \oD K
\end{equation*}
for every $t\in\R$ and $K\in\Ko$. Moreover,  $\lim_{t\to+\infty} \psi(t)=0$.
\end{lemma}
\begin{proof}
Let $\psi_1,\ldots,\psi_4$ be as in Lemma \ref{le:covariant_growth_functions}. By Lemma \ref{le:covariant_growth_functions_2}, there exist constants $c_3,c_4$ such that $\psi_3(t)\equiv c_3$ and $\psi_4(t)\equiv c_4$. Moreover, $\psi_1$ and $\psi_2$ are non-negative and only differ by a constant. Hence, it suffices to show that $\lim_{t\to+\infty} \psi_1(t)=\lim_{t\to+\infty} \psi_2(t)=0$ and $c_3=c_4=0$. To show this, let $r,b>0$ and let $v_r^b\in\CV$ be defined by $\epi v_r^b = \epi \ell_{T_r} \cap \{x_1\leq b\}$, where $T_r$ is defined as in Lemma \ref{le:t_l_h}. Note, that $\elim_{b\to+\infty}v_r^b =\ell_{T_r}$. Let $\tau_b$ be the translation $x\mapsto x+be_1$ and set $\ell_r^b:= \ell_{T_r}\circ \tau_b^{-1}+\tfrac br$. Then
\begin{equation*}
v_r^b \mn \ell_r^b = \ell_{T_r},\qquad \dom (v_r^b \vee \ell_r^b) \subset \{x_1=b\}.
\end{equation*}
Thus, by the valuation property and Lemma \ref{le:covariant_valuations_simple}, we obtain
\begin{equation*}
h(\oZ(v_r^b),e_1)=h(\oZ(\ell_{T_r}),e_1)-h(\oZ(\ell_r^b),e_1).
\end{equation*}
Using the translation invariance and continuity of $\oZ$ now gives
\begin{equation*}
r \psi_1(0) + r^2 \frac{c_3+c_4}{(n+1)!}=h(\oZ(\ell_{T_r}),e_1) = \lim_{b\to+\infty} h(\oZ(v_r^b),e_1) = \lim_{b+\infty} r (\psi_1(0)-\psi_1(\tfrac b r))
\end{equation*}
for every $r >0$. Hence, $\lim_{t\to+\infty} \psi_1(t)=0$ and $c_3+c_4=0$. Similarly, evaluating the support functions at $-e_1$ gives $\lim_{t\to+\infty} \psi_2(t)=0$ and $c_3-c_4=0$. Consequently, $c_3=c_4=0$.
\end{proof}

By Lemma \ref{le:reduction}, we obtain the following result as an immediate corollary of the last result. 

\begin{lemma}\label{co_trivial}
Every  continuous, increasing, $\sln\!$ covariant, translation invariant Minkowski valuation on $\CV$ is trivial.\end{lemma}

For a given continuous, $\sln$ covariant and translation invariant Minkowski valuation $\oZ:\CV\to \Kn$, we call the function $\psi$ from Lemma \ref{le:covariant_growth_functions_3} the \textit{cone growth function} of $\oZ$.

\begin{lemma}
\label{le:covariant_growth_functions_4}
If the operator $\,\oZ:\CV\to\Kn$ is a continuous, $\sln$ covariant and  translation invariant Minkowski valuation with cone growth function $\psi$,  then $\psi$ is decreasing and
\begin{equation*}
\oZ(\Ind_K+t)=-\psi'(t) \oD K
\end{equation*}
for every $t\in\R$ and $K\in\Ko$.
\end{lemma}
\begin{proof}
Let $\zeta$ be as in Lemma \ref{le:covariant_growth_functions}. Since $\zeta\geq 0$, it suffices to show that $\zeta=-\psi'$. Therefore, for $h>0$ let $u_h\in\CV$ be defined by $\epi u_h =  \epi \ell_{[0,e_1/h]} \cap \{x_1 \leq 1\}$. By Lemma~\ref{th:epi_lvl_sets},  we have $\elim_{h\to 0} u_h = \Ind_{[0,e_1]}$. Denote by $\tau$ the translation $x\mapsto x+e_1$ and define $\ell_h=\ell_{[0,e_1/h]}\circ \tau^{-1}+h$. Then,
\begin{equation*}
u_h \mn \ell_h =  \ell_{[0,e_1/h]},\qquad u_h \vee \ell_h = \Ind_{\{e_1\}}+h.
\end{equation*}
Hence, by the properties of $\oZ$ and the definitions of $\psi$ and $\zeta$ this gives
\begin{equation*}
\zeta(t)=h(\oZ(\Ind_{[0,e_1]}+t),e_1)=\lim_{h\to 0^+} h(\oZ(u_h+t),e_1) = \lim_{h\to 0^+} \frac{\psi(t)-\psi(t+h)}{h}
\end{equation*}
for every $t\in\R$. The claim follows, since $\psi$ is differentiable.
\end{proof}

The function $\zeta=-\psi'$ appearing in the above Lemma is called the \textit{indicator growth function} of $\oZ$. 
Lemma \ref{le:covariant_growth_functions_2} shows that the indicator growth function $\zeta$ of a continuous, $\sln$ covariant and trans\-lation invariant Minkowski valuation $\oZ$ determines its cone growth function $\psi$ up to a constant. Since $\lim_{t\to\infty} \psi(t)=0$,  the constant is also  determined by $\zeta$. Thus $\psi$  is completely determined by the indicator growth function of $\oZ$ and  Lemma \ref{le:reduction} implies the following result.

\begin{lemma}
\label{le:reduction2co}
Every continuous, $\sln$\! covariant, translation invariant Minkowski valuation on $\CV$ is uniquely determined by its indicator growth function.
\end{lemma}

\subsection{Proof of Theorem  \ref{covariant}}

By Lemma \ref{le:diff_body_is_a_val}, for $\zeta\in\Dz$, the operator $u\mapsto \oD\oid{\zeta\circ u}$  defines a continuous, decreasing,  $\sln$ covariant and  translation invariant Minkowski valuation on $\CV$.

Conversely, let now a continuous, monotone,  $\sln$ covariant and  translation invariant Min\-kowski valuation $\oZ$ be given  and let  $\zeta$ be its indicator growth function. Lemma \ref{co_trivial} implies that we may assume that $\oZ$ is decreasing. By Lemma \ref{le:reduction2co}, the valuation $\oZ$ is uniquely determined by $\zeta$. 
For $P=[0,e_1]\in\Po$, we have
\begin{equation*}
h(\oZ(\Ind_{P}+t),e_1)=\zeta(t)\,h(\oD P,e_1)=\zeta(t)
\end{equation*}
for every $t\in\R$. Since $\oZ$ is decreasing, also  $\zeta$ is decreasing. Since $\zeta=-\psi'$, it follows from Lemma \ref{le:covariant_growth_functions_2} that
$$\int_0^\infty \zeta(t)=\psi(0)- \lim_{t\to \infty} \psi(t)=\psi(0).$$
Thus $\zeta\in\D{0}$. 

\goodbreak
For $u=\Ind_P+t$ with arbitrary $P\in\Po$ and $t\in\R$, we have
\begin{equation*}
h(\oD \oid{\zeta\circ u},z) = \int_0^{+\infty} h(\oD\{\zeta\circ u\geq s\},z) \d s = \zeta(t)\, h(\oD P,z)
\end{equation*}
for every $z\in\sn$. Hence $\oD \oid{\zeta\circ(\Ind_P+t)}= \zeta(t) \oD P$ for $P\in\cP_0^n$ and $t\in\R$. By  Lemma \ref{le:diff_body_is_a_val},
$$u\mapsto \oD \oid{\zeta\circ u}$$
defines a continuous, decreasing,  $\sln$ covariant and  translation invariant Min\-kowski valuation  on $\CV$
 and  $\zeta$ is its indicator growth function.
Thus Lemma 
\ref{le:reduction2co} completes the proof of the theorem.

\subsection*{Acknowledgments}
The work of Monika Ludwig and Fabian Mussnig was supported, in part, by Austrian Science Fund (FWF) Project P25515-N25.  The work of Andrea Colesanti 
was supported by the G.N.A.M.P.A. and by the F.I.R. project 2013: Geometrical and Qualitative Aspects of PDE's.

\small

\begin{thebibliography}{10}

\bibitem{AbardiaWannerer}
J.~Abardia and T.~Wannerer, {\em Aleksandrov-{F}enchel inequalities for unitary
  valuations of degree 2 and 3}, Calc. Var. Partial Differential Equations {\bf
  54} (2015),  1767--1791.

\bibitem{Alesker99}
S.~Alesker, {\em Continuous rotation invariant valuations on convex sets}, Ann.
  of Math. (2) {\bf 149} (1999), 977--1005.

\bibitem{Alesker01}
S.~Alesker, {\em Description of translation invariant valuations on convex sets
  with solution of {P}.\ {M}c{M}ullen's conjecture}, Geom. Funct. Anal. {\bf 11}
  (2001), 244--272.

\bibitem{Alesker_convex}
S.~Alesker, {\em Valuations on convex functions and convex
sets and Monge-Amp\`ere operators}, Preprint (arXiv:1703.08778).


\bibitem{AMJV}
D.~Alonso-Guti\'errez, B.~Gonz\'alez~Merino, C.~H. Jim\'enez, and R.~Villa,
  {\em John's ellipsoid and the integral ratio of a log-concave function},
  J. Geom. Anal., in press.

\bibitem{AmbrosioFuscoPallara}
L.~Ambrosio, N.~Fusco, and D.~Pallara, {\em Functions of {B}ounded {V}ariation
  and {F}ree {D}iscontinuity {P}roblems}, Oxford Mathematical Monographs, The
  Clarendon Press Oxford University Press, New York, 2000.

\bibitem{BaryshnikovGhristWright}
Y.~Baryshnikov, R.~Ghrist, and M.~Wright, {\em Hadwiger's {T}heorem for
  definable functions}, Adv. Math. {\bf 245} (2013), 573--586.

\bibitem{Bernig:Fu}
A.~Bernig and J.~H.~G. Fu, {\em Hermitian integral geometry}, Ann. of Math. (2)
  {\bf 173} (2011), 907--945.

\bibitem{BobkovColesantiFragala}
S.~G. Bobkov, A.~Colesanti, and I.~Fragal{\`a}, {\em Quermassintegrals of
  quasi-concave functions and generalized {P}r\'ekopa-{L}eindler inequalities},
  Manuscripta Math. {\bf 143} (2014), 131--169.

\bibitem{CavallinaColesanti}
L.~Cavallina and A.~Colesanti, {\em Monotone valuations on the space of convex
  functions}, Anal. Geom. Metr. Spaces {\bf 3} (2015), 167--211.

\bibitem{CLYZ}
A.~Cianchi, E.~Lutwak, D.~Yang, and G.~Zhang, {\em Affine {M}oser-{T}rudinger
  and {M}orrey-{S}obolev inequalities}, Calc. Var. Partial Differential
  Equations {\bf 36} (2009), 419--436.

\bibitem{ColesantiFragala}
A.~Colesanti and I.~Fragal{\`a}, {\em The first variation of the total mass of
  log-concave functions and related inequalities}, Adv. Math. {\bf 244} (2013),
  708--749.

\bibitem{ColesantiLombardi}
A.~Colesanti and N.~Lombardi, {\em Valuations on the space of quasi-concave
  functions}, Geo\-metric Aspects of Functional Analysis,  71--105, Lecture Notes in Mathematics 2169,  Springer International Publishing, Cham, 2017.

\bibitem{ColesantiLombardiParapatits}
A.~Colesanti, N.~Lombardi and L.~Parapatits, {\em Translation invariant valuations on quasi-concave
  functions}, Preprint (arXiv:1703.06867).

\bibitem{ColesantiLudwigMussnig}
A.~Colesanti, M.~Ludwig, and F.~Mussnig, {\em Valuations on convex functions},
 Preprint (arXiv:1703.06455).

\bibitem{Haberl_sln}
C.~Haberl, {\em Minkowski valuations intertwining with the special linear
  group}, J. Eur. Math. Soc. (JEMS) {\bf 14} (2012), 1565--1597.

\bibitem{Haberl:Parapatits_centro}
C.~Haberl and L.~Parapatits, {\em The centro-affine {H}adwiger theorem}, J.
  Amer. Math. Soc. {\bf 27} (2014), 685--705.

\bibitem{Haberl:Parapatits_crelle}
C.~Haberl and L.~Parapatits, {\em Valuations and surface area measures}, J.
  Reine Angew. Math. {\bf 687} (2014), 225--245.

\bibitem{HaberlParapatits_moments}
C.~Haberl and L.~Parapatits, {\em Moments and valuations}, Amer. J. Math. {\bf
  138} (2017), 1575--1603.

\bibitem{HaberlSchuster09}
C.~Haberl and F.~Schuster, {\em Asymmetric affine {$L\sb p$} {S}obolev
  inequalities}, J. Funct. Anal. {\bf 257} (2009), 641--658.

\bibitem{Haberl:Schuster:Xiao}
C.~Haberl, F.~Schuster, and J.~Xiao, {\em An asymmetric affine
  {P}\'olya-{S}zeg\"o principle}, Math. Ann. {\bf 352} (2012), 517--542.

\bibitem{Hadwiger:V}
H.~Hadwiger, {\em {V}or\-lesungen \"uber {I}nhalt, {O}ber\-fl\"ache und
  {I}so\-peri\-metrie}, Springer, Berlin, 1957.

\bibitem{HLYZ_acta}
Y.~Huang, E.~Lutwak, D.~Yang, and G.~Zhang, {\em Geometric measures in the dual
  {B}runn--{M}inkowski theory and their associated {M}inkowski problems}, Acta
  Math. {\bf 216} (2016), 325--388.

\bibitem{Klain:Rota}
D.~A. Klain and G.-C. Rota, {\em Introduction to {G}eometric {P}robability},
  Cambridge University Press, Cambridge, 1997.

\bibitem{Kone}
H.~Kone, {\em Valuations on {O}rlicz spaces and {$L\sp \phi$}-star sets}, Adv.
  in Appl. Math. {\bf 52} (2014), 82--98.

\bibitem{LiLeng_polytopes}
J.~Li and G.~Leng, {\em {$L_p$} {M}inkowski valuations on polytopes}, Adv.
  Math. {\bf 299} (2016), 139--173.

\bibitem{LiMa}
J.~Li and D.~Ma, {\em Laplace transforms and valuations}, J. Funct. Anal. {\bf
  272} (2017), 738--758.

\bibitem{Ludwig:projection}
M.~Ludwig, {\em Projection bodies and valuations}, Adv. Math. {\bf 172} (2002),
  158--168.

\bibitem{Ludwig:Minkowski}
M.~Ludwig, {\em Minkowski valuations}, Trans. Amer. Math. Soc. {\bf 357}
  (2005), 4191--4213.

\bibitem{Ludwig:convex}
M.~Ludwig, {\em Minkowski areas and valuations}, J. Differential Geom. {\bf 86}
  (2010), 133--161.

\bibitem{Ludwig:Fisher}
M.~Ludwig, {\em Fisher information and valuations}, Adv. Math. {\bf 226}
  (2011), 2700--2711.

\bibitem{Ludwig:survey}
M.~Ludwig, {\em Valuations on function spaces}, Adv. Geom. {\bf 11} (2011),
  745--756.

\bibitem{Ludwig:sobval}
M.~Ludwig, {\em Valuations on {S}obolev spaces}, Amer. J. Math. {\bf 134}
  (2012), 827--842.

\bibitem{Ludwig:MM}
M.~Ludwig, {\em Covariance matrices and valuations}, Adv. in Appl. Math. {\bf
  51} (2013), 359--366.

\bibitem{Ludwig:Reitzner2}
M.~Ludwig and M.~Reitzner, {\em A classification of \,{${\rm SL}(n)$\!} invariant
  valuations}, Ann. of Math. (2) {\bf 172} (2010), 1219--1267.

\bibitem{LXZ}
M.~Ludwig, J.~Xiao, and G.~Zhang, {\em Sharp convex {L}orentz-{S}obolev
  inequalities}, Math. Ann. {\bf 350} (2011), 169--197.

\bibitem{LYZ2002b}
E.~Lutwak, D.~Yang, and G.~Zhang, {\em Sharp affine ${L}_p$ {S}obolev
  inequalities}, J. Differential Geom. {\bf 62} (2002), 17--38.

\bibitem{LYZ2006}
E.~Lutwak, D.~Yang, and G.~Zhang, {\em Optimal {S}obolev norms and the {$L\sp
  p$} {M}inkowski problem}, Int. Math. Res. Not. {\bf 62987} (2006), 1--20.

\bibitem{Ma2016}
D.~Ma, {\em Real-valued valuations on {S}obolev spaces}, Sci. China Math. {\bf
  59} (2016), 921--934.

\bibitem{NatansonII}
I.~P. Natanson, {\em Theory of Functions of a Real Variable. {V}ol. {II}},
  Translated from the Russian by Leo F. Boron, Frederick Ungar Publishing Co.,
  New York, 1961.

\bibitem{Ober2014}
M.~Ober, {\em {$L\sb p$}-{M}inkowski valuations on {$L\sp q$}-spaces}, J. Math.
  Anal. Appl. {\bf 414} (2014), 68--87.

\bibitem{RockafellarWets}
R.~T. Rockafellar and R.~J.-B. Wets, {\em Variational {A}nalysis}, Grundlehren
  der Mathe\-matischen Wissen\-schaften, vol. 317, Springer-Verlag, Berlin,
  1998.

\bibitem{Rudin:RCA}
W.~Rudin, {\em Real and Complex Analysis}, Third ed., McGraw-Hill Book Co., New
  York, 1987.

\bibitem{Schneider:CB2}
R.~Schneider, {\em Convex {B}odies: the {B}runn-{M}inkowski {T}heory}, {S}econd
  expanded ed., Encyclopedia of Mathe\-matics and its Applications, vol. 151,
  Cambridge University Press, Cambridge, 2014.

\bibitem{Schuster:Wannerer}
F.~Schuster and T.~Wannerer, {\em {${\rm GL}(n)$\!} contravariant {M}inkowski
  valuations}, Trans. Amer. Math. Soc. {\bf 364} (2012), 815--826.

\bibitem{Tsang:Lp}
A.~Tsang, {\em Valuations on ${L}^p$ spaces}, Int. Math. Res. Not. {\bf 20}
  (2010), 3993--4023.

\bibitem{Tsang:Minkowski}
A.~Tsang, {\em Minkowski valuations on ${L}^p$-spaces}, Trans. Amer. Math. Soc.
  {\bf 364} (2012), 6159--6186.

\bibitem{Tuo_Wang}
T.~Wang, {\em The affine {S}obolev-{Z}hang inequality on {BV}$({\R}^n)$}, Adv.
  Math. {\bf 230} (2012), 2457--2473.

\bibitem{Tuo_Wang:PSP}
T.~Wang, {\em The affine {P}\'olya-{S}zeg\"o principle: equality cases and
  stability}, J. Funct. Anal. {\bf 265} (2013), 1728--1748.

\bibitem{Tuo_Wang_semi}
T.~Wang, {\em Semi-valuations on {${\rm BV}(\mathbb R\sp n)$}}, Indiana Univ.
  Math. J. {\bf 63} (2014), 1447--1465.

\bibitem{Wannerer2011}
T.~Wannerer, {\em {${\rm GL}(n)$\!} equivariant {M}inkowski valuations}, Indiana
  Univ. Math. J. {\bf 60} (2011), 1655--1672.

\bibitem{Zhang99}
G.~Zhang, {\em The affine {S}obolev inequality}, J. Differential Geom. {\bf 53}
  (1999), 183--202.

\end{thebibliography}

\bigskip\bigskip\small
\parindent 0pt

\parbox[t]{8.5cm}{Andrea Colesanti\\
Dipartimento di Matematica e Informatica \lq\lq U. Dini\rq\rq\\
Universit\`a degli Studi di Firenze \\
Viale Morgagni 67/A\\
50134, Firenze, Italy\\
e-mail:  colesant@math.unifi.it
}
\parbox[t]{8cm}{Monika Ludwig\\
Institut f\"ur Diskrete Mathematik und Geometrie\\
Technische Universit\"at Wien\\
Wiedner Hauptstra\ss e 8-10/1046\\
1040 Wien, Austria\\
e-mail: monika.ludwig@tuwien.ac.at}

\bigskip\smallskip

\parbox[t]{8.5cm}{
Fabian Mussnig\\
Institut f\"ur Diskrete Mathematik und Geometrie\\
Technische Universit\"at Wien\\
Wiedner Hauptstra\ss e 8-10/1046\\
1040 Wien, Austria\\
e-mail: fabian.mussnig@tuwien.ac.at}

\end{document}